\documentclass[10 pt]{amsart}

\usepackage{amsmath} 
\usepackage{amsfonts}
\usepackage{amssymb}
\usepackage{amstext}
\usepackage{amsbsy}
\usepackage{amsopn}
\usepackage{amsthm}
\usepackage{amsxtra}
\usepackage{graphicx}
\usepackage{hyperref}

\newtheorem{theorem}{Theorem}[section]
\newtheorem{lemma}[theorem]{Lemma}
\newtheorem{proposition}[theorem]{Proposition}
\newtheorem{corollary}[theorem]{Corollary}
\newtheorem{definition}[theorem]{Definition}
\newtheorem{question}[theorem]{Question}

\newcommand{\R}{\mathcal{R}}
\newcommand{\Z}{\mathbb{Z}}
\newcommand{\Q}{\mathcal{Q}}
\newcommand{\N}{\mathbb{N}}

\newcommand{\A}{\mathcal{A}}
\newcommand{\CM}{\mathcal{M}}
\newcommand{\CL}{\mathcal{L}}

\newcommand{\CP}{\mathcal{P}}
\renewcommand{\P}{\mathcal{P}}

\newcommand{\one}{\boldsymbol{1}}

\newcommand{\Aut}{\mathrm{Aut}}
\newcommand{\id}{\mathrm{Id}}
\newcommand{\tp}{\mathrm{top}}
\newcommand{\LS}{\mathrm {LS}}
\newcommand{\proj}{\mathrm {proj}}

\title{Characteristic measures for language stable subshifts}
\author{Van Cyr}
\address{Bucknell University, Lewisburg, PA 17837 USA}
\email{van.cyr@bucknell.edu}
\author{Bryna Kra}
\address{Northwestern University, Evanston, IL 60208 USA}
\email{kra@math.northwestern.edu}

\subjclass[2010]{37B10, 68R15, 37A15}
\keywords{subshift, block complexity}

\thanks{The second author was partially supported by NSF grant DMS-1800544.}

\begin{document}

\begin{abstract}
We consider the problem of when a symbolic dynamical system supports a Borel probability measure that is invariant under every element of its automorphism group.  It follows readily from a classical result of Parry that 
the full shift on finitely many symbols, and more generally any mixing subshift of finite type, supports such a measure.  Frisch and Tamuz recently dubbed such measures characteristic, 
and further showed that every zero entropy subshift has a characteristic measure.  
While it remains open if every subshift over a finite alphabet has a characteristic measure, 
we define a new class of shifts, which we call language stable subshifts, 
and show that these shifts have characteristic measures.  This is a large 
class that is generic in several senses and contains numerous positive entropy examples. 
\end{abstract}

\maketitle

\section{Introduction}
Suppose $X$ is a compact metric space and $T\colon X\to X$ is a homeomorphism.  The 
Krylov-Bogolioubov 
Theorem says that there exists a Borel probability measure $\mu$ that is supported on $X$ and is invariant under $T$, meaning that $T_*\mu=\mu$. 
Using tools of ergodic theory to study the measure-preserving system $(X,T,\mu)$, 
we can often obtain information about the topological dynamical system $(X,T)$. 
While the homeomorphism $\sigma$ determines a $\Z$-action on $X$, 
there is another (often larger) action on $X$ determined by the automorphism group 
$$ 
\Aut(X):=\{\psi\in\text{Homeo}(X)\colon\psi T=T\psi\} 
$$ 
of all self-conjugacies of $(X,T)$.  This is a natural approach for 
studying when two topological systems $(X,T)$ and $(Y,S)$ are topologically conjugate: if $(X,T)$ 
and $(Y,S)$ are  conjugate as $\Z$-systems, then $\Aut(X)\cong\Aut(Y)$ as 
groups, 
 and $(X,\Aut(X))$ and $(Y,\Aut(Y))$ are topologically conjugate as $\Aut(X)$-systems.  Systems that can be distinguished as $\Aut(X)$-systems can 
therefore also be distinguished as $\Z$-systems. 
 
 With this in mind, it is natural to ask if there is an analog of the Krylov-Bogolioubov Theorem for the action of the automorphism group of a dynamical system.  
Using terminology introduced by Frisch and Tamuz~\cite{FT2}, if $(X,T)$ is a topological dynamical system and $\Aut(X)$ denotes its automorphism group, a Borel probability measure $\mu$ is {\em characteristic} for $(X,\sigma)$ if $\psi_*\mu=\mu$ for all $\psi\in\Aut(X)$.   They~\cite[Question 1.3]{FT2} ask: does every subshift (over a finite alphabet) have a characteristic measure?

It is natural to restrict this question to studying particular systems, rather than consider an arbitrary topological dynamical system; for example, 
given a Cantor set it is easy to check that there is no Borel probability measure that is invariant under all self homeomorphisms of the space to itself.  However, for the broad class of subshifts, the question remains open.  For these systems, the classical theorem of Curtis, Hedlund, and Lyndon (see~\cite{hedlund}) states that every automorphism is given as a sliding 
block code, and in particular it follows that there are only countably many automorphisms.  However, the group of automorphisms can be quite complicated:  for example (see~\cite{hedlund, BLR}), the automorphism group of 
any mixing shift of finite type contains isomorphic copies of any finite group, the free group on any number of generators, along with copies of many other known groups.  
A natural way to attempt to answer Frisch and Tamuz's question is to find 
all of the automorphisms of a subshift and all of the ergodic measures, and then check if any of these measures is a characteristic measure.  Unfortunately, this is not a practical method, as many easy to state questions even about full shifts remain open.  For example, we do not have enough 
information about the automorphism group of a full shift to determine whether the automorphism group of the full shift on two symbols is isomorphic to that on three symbols.  While this type of strategy is doomed to failure with current tools, there are 
other methods for showing the existence of a characteristic measure without knowing either what all of the automorphisms are, or much about the simplex of invariant measures, $\mathcal{M}(X)$. 

We defer the precise definitions and explanations of these results 
until Section~\ref{sec:background}, but give a quick summary of four currently known methods for 
proving that a subshift $(X, \sigma)$ has a characteristic measure: 
	\begin{enumerate}
	\item If $\Aut(X)$ is amenable, this follows from the Krylov-Bogolioubov 

Theorem.\label{it:one}
	\item \label{it:two} 
	When the topological entropy $h_{\tp}(X)$ is $0$, this is the main result in Frisch and Tamuz~\cite{FT2}.  
	\item \label{it:three}
	If there exists a $\sigma$-invariant probability measure $\mu$ supported 
on $X$ such that 
	$$ 
	\{\nu\in\mathcal{M}(X)\colon(X,\sigma,\mu)\text{ is measurably isomorphic to }(X,\sigma,\nu)\}
	$$ 
	is finite, then $X$ has a characteristic measure. 
	\item \label{it:four}
	If there exists a closed subshift $Y\subseteq X$ that supports an $\Aut(Y)$-characteristic measure and 
	there are only finitely many $Z\subseteq X$ such that $(Y,\sigma)$ is topologically 
	conjugate to $(Z,\sigma)$, then $X$ has a characteristic measure.
	\end{enumerate} 
We discuss the first two methods further in Section~\ref{sec:background}; 
properties~\eqref{it:three} and~\eqref{it:four} are implicit in the literature and are described in Lemmas~\ref{lem:measurable} and~\ref{lem:topological}.   
In light of Frisch and Tamuz's theorem, listed as the method~\eqref{it:two}, 
determining whether every symbolic system has a characteristic measure comes down to developing techniques for finding characteristic measures in symbolic systems of positive entropy. 

Our main result gives a large class of symbolic systems, containing numerous systems of positive entropy, that admit characteristic measures.  To define this class, we introduce a new condition called language stability 
(see Section~\ref{section:main}).  An arbitrary shift can be approximated from the outside by shifts of finite type, and the  class of language stable shifts are those that are approximated, infinitely often, faster than the expected rate. 
This definition is characterized by considering the words that do not appear in the language of the system, and placing restrictions on the frequency with which such  new forbidden words arise.  
The class of language stable systems includes many well-studied systems, including numerous zero entropy systems such as the Sturmian systems, and numerous positive entropy systems, such as the shifts of finite type.  However this class also contains many more systems:
in Section~\ref{sec:generic} we show that the set of language stable subshifts is generic, 
in a strong sense.  We show that the language stable shifts are 
a dense $G_{\delta}$, with respect to the Hausdorff topology, in both the 
space of all subshifts with a fixed alphabet and in the subspace of positive entropy subshifts with a fixed alphabet.

Our first main result is to show that language stability guarantees the existence of a characteristic measure: 
\begin{theorem}\label{thm:main} 
Every language stable subshift supports a characteristic measure. 
\end{theorem} 
In fact we deduce (see Corollary~\ref{cor:language-stable}) that the characteristic measure we find 
is a measure of maximal entropy for the subshift.

The class of language stable shifts has numerous properties that are interesting in their own right,
 and in this work we focus on their relevancy to finding a characteristic measure on a subshift.  
 In addition to showing that language stable shifts are a generic class of subshifts with characteristic measures, 
 we explicitly construct examples of systems 
with characteristic measures that were not previously known to exist. 
More precisely, in Section~\ref{sec:example} we  build an example of a language stable subshift, which thus has a characteristic measure by our theorem, but none of the four previously 
known methods of proving the existence of characteristic measures 
applies in this case. 
This example is a product of three systems, and the properties of these systems are 
of independent interest.  
The first component system $X$ is a language stable, positive entropy, minimal subshift, and this is constructed by inductively defining systems that approximate the system $X$ from the outside while controlling the the language of the system.  The second component system $Y$ is a full shift, which by a result of Boyle, Lind, and Rudolph~\cite{BLR} has a unique characteristic measure of positive entropy (and this is the measure of maximal entropy).  The third component $Z$ is a language stable shift with countably many ergodic measures, all of which are isomorphic, and it is constructed by coding a fixed rotation with respect to different partitions that converge to the coding that gives rise to the Sturmian coding.

In our example,  the automorphism group of this system is not amenable, the system has positive entropy, 
and every measure supported on the system is measurably isomorphic to infinitely many other measures on the system, ensuring 
that none of the first three conditions is satisfied.  The fourth condition is more subtle.  We do not rule out the possibility that it could be used to show that our example carries some characteristic measure.  However, Corollary~\ref{cor:language-stable} guarantees that our example carries a characteristic measure of maximal entropy and we show that this measure could not arise by applying the fourth condition to our example.  More specifically, we show that any proper subshift of our example that has full topological entropy is topologically conjugate to infinitely many other proper subshifts of our example.  This ensures that, if the fourth condition does apply to our example, it can only produce characteristic measures with less than maximal entropy.
Therefore our result guarantees the existence of a characteristic measure 
of maximal entropy that could not have been seen to exist by any of the four previously described methods.
Further, we emphasize that our theorem shows the existence of a characteristic measure in our system without our having to either determine the algebraic structure of its group of automorphisms or describe all of its invariant measures.  This is highly advantageous: even if we could describe the automorphism group of each of 
the three systems, the automorphism group of their product may be significantly more complicated than just the product of their individual automorphism groups.  By way of example, we mention a recent theorem of Salo and 
Schraudner~\cite{salo-schraudner} that if $X\subseteq\{0,1\}^{\Z}$ is the 
``sunny side up shift'' consisting of all elements of $\{0,1\}^{\Z}$ that 
have at most one $1$, then $\Aut(X)\cong\Z$ while $\Aut(X\times X)\cong(\Z^{\infty}\rtimes S_{\infty})\rtimes(\Z^2\rtimes S_2)$.

We conclude this introduction with an application to further motivate the 
study of characteristic measures.  Beyond existence of a characteristic measure for a subshift, 
it is a natural question whether knowledge of such a measure gives practical information about the subshift.  
There has been significant interest in determining the algebraic properties of $\Aut(X)$ for different subshifts $X$.  Most of these advances begin with assumptions on $X$, such as a constraint on 
some type of complexity of the shift, like the growth rate of the block complexity function or the visiting complexity for recurrence, or 
structure in the dynamics of the shift, such as being a shift of finite type or being a Toeplitz shift. 
One then uses these constraints to show that $\Aut(X)$ either must have or cannot have certain algebraic properties.  Another, less explored, way to study $\Aut(X)$ on some explicitly given shift is to find small range block codes that define automorphisms and study the relations in the subgroup of $\Aut(X)$ that they generate.  This approach seems fruitful both because of its computational nature and because there are many natural questions about $\Aut(X)$ that, so far, have resisted being answered with previously developed approaches.  As a specific example, it is unknown whether there exists 
a subshift whose automorphism group contains a finitely generated, nonabelian, infinite, nilpotent group. 
One could imagine a computational approach to this problem where a specific subshift $(X,\sigma)$ is selected, the block codes of small range are checked to see which ones determine automorphisms, and it is checked whether, among them, there exist $\varphi,\psi\in\Aut(X)$ of infinite order
such that 
$$
\theta = [\varphi, \psi] = \varphi\psi\varphi^{-1}\psi^{-1}\neq\id, \quad \varphi\theta = \theta\varphi, \quad\text{ and } \quad \psi\theta = \theta\psi.
$$
Then the group generated by these relations would be an infinite, nonabelian, quotient of the discrete Heisenberg group. 
For explicitly given block codes, we note that these relations can be checked easily simply by finding the block code each of these group elements 
defines and checking to see if it is the identity (a different approach would be needed to check if $\varphi$ and $\psi$ have infinite order).  A challenge to investigating the structure of $\Aut(X)$ in this way, however, is that we must find block codes that determine automorphisms of $X$.  
This leads to the question: 
\begin{question}
Given a subshift $X$ and a block code $\varphi$ defined on $X$, determine 
whether $\varphi$ determines an automorphism of $X$. 
\end{question}

In Section~\ref{sec:application}, we explain how a characteristic measure 
gives rise to a set of nontrivial conditions on a block code that are necessary for it to be invertible.  We give an example where they can be used to show a certain block code does not have an inverse.  Although these conditions are not sufficient for a block code to have an inverse, they do allow one to eliminate many block codes as candidates for being automorphisms.

\subsection*{Acknowledgment}
We thank Anthony Quas for helpful conversations during the preparation of 
this article.

\section{Preliminaries and Notation}
\label{sec:background}

\subsection{Upper Banach density} If $S\subseteq\N$, then the {\em upper Banach density $d^*(S)$} of $S$ is defined to be  
$$ 
d^*(S):=\limsup_{n\to\infty}\max_k\frac{|S\cap\{k,k+1,k+2,\dots,k+n-1\}|}{n}. 
$$ 
A set $S$ has upper Banach density $1$ if and only if there are arbitrarily long runs of consecutive integers that are all elements of $S$. 

\subsection{Symbolic systems} Let $\A$ be a finite set and let $\A^{\Z}$ be the set of functions $x\colon\Z\to\A$, and denote $x\in\A^\Z$ as $x = 
(x_i)_{i\in\Z}$.  The space $\A^Z$ is a compact metric space when endowed 
with the metric 
$$ 
d\bigl((x_i),(y_i)\bigr):=2^{-\inf\{|i|\colon x_i\neq y_i\}}. 
$$ 
The {\em left-shift} map $\sigma\colon\A^{\Z}\to\A^{\Z}$ defined by $(\sigma x)_i:=x_{i+1}$ for all $i\in\Z$ is a homeomorphism.  
For each $w=(w_0,\dots,w_{n-1})\in\A^n$ the {\em cylinder set} is 
$$ 
[w]_0^+:=\{x\in\A^{\Z}\colon x_i=w_i \text{ for all }0\leq i<n\} 
$$ 
and the collection of sets $\{\sigma^i[w]_0^+\colon w\in\A^*\text{, }i\in\Z\}$ gives a basis for the topology of $\A^{\Z}$ (when no confusion can arise, we use the simpler notation $[w]$ to mean $[w]_0^+$). 
If $X\subset\A^\Z$ is closed and $\sigma$-invariant, then $(X, \sigma)$ is
a  {\em subshift}, and when the shift $\sigma$ is understood from the context, we omit the transformation from the notation and refer to $X\subset 
\A^\Z$ as a subshift. 

The {\em language} of a subshift $(X, \sigma)$ is 
$$ 
\mathcal{L}(X):=\{w\in\A^*\colon[w]_0^+\cap X\neq\emptyset\}
$$ 
and any $w\in\mathcal{L}(X)$ is called a {\em word} in the language ($w\in\mathcal{L}(X)$ is sometimes referred to as a {\em factor} in the literature).  
The {\em complexity} $P_X(n)$ counts the number of words of length $n$ in 
the language $\CL(X)$.

$\mathrm{Homeo}(X)$ forms a group under composition and $\Aut(X)$ is the centralizer of $\sigma$ in $\mathrm{Homeo}(X)$.  
A map $\psi\colon X\to \A^\Z$ is called a {\em sliding block code} 
if there exists $R\in\N$ and a map $\Psi\colon\mathcal{L}_{2R+1}(X)\to\A$ 
such that for all $x\in X$ and $i\in\Z$, we have 
$$ 
(\psi x)_i=\Psi(x_{i-R},\dots,x_i,\dots,x_{i+R}). 
$$ 
In this case the number $R$ is called a {\em range} for $\psi$.  A classical result characterizes the automorphisms of a subshift: 
\begin{theorem}[Curtis-Hedlund-Lyndon Theorem~\cite{hedlund}]
\label{th:CHL}
Every element of $\Aut(X)$ is a sliding block code.  Conversely, any sliding block code from $X$ to $X$ that has a sliding block code inverse is an element of $\Aut(X)$. 
\end{theorem}

Following Frisch and Tamuz~\cite{FT2}, we define the central object of study:
\begin{definition} 
Let $(X,T)$ be a topological dynamical system and let 
$$ 
\Aut(X):=\{\psi\in\text{Homeo}(X)\colon\psi T=T\psi\} 
$$ 
be its automorphism group.  A Borel probability measure $\mu$ is {\em characteristic} for $(X,T)$ if $\psi_*\mu=\mu$ for all $\psi\in\Aut(X)$. 
\end{definition} 
We study this notion for a subshift $(X,\sigma)$.

\subsection{Coding of orbits}
Assume $(X, T)$ is an invertible topological dynamical system and let $\CP$ be a partition of the space $X$ into finitely many sets, meaning that $\CP = \{P_1, \ldots, P_n\}$ for some sets $P_i\subset X$ satisfying 
$\bigcup_{i=1}^nP_i = X$.  Note that we make no assumption that the sets $P_i$ are open and require that their union cover all of $X$.  If $x\in X$, then the {\em coding of the orbit of $x$} is the sequence $(x_j)_{j\in\Z}$ defined by $x_j = i$ if and only if $T^jx\in P_i$.  The {\em coding of the system $(X,T)$} is the symbolic system obtained by taking the closure of the codings of all $x\in X$, and it is easy to check that this is a subshift of $\{1, \ldots, n\}^\Z$.

\subsection{The forbidden word construction} 
We recall how to construct subshifts by specifying a list of forbidden words. 
Let $\mathcal{F}\subseteq\A^*$ be a (finite or infinite) set of words.  Define 
$$ 
Y_{\mathcal{F}}:=\{y\in\A^{\Z}\colon\sigma^iy\notin[w]_0^+\text{ for all }w\in\mathcal{F}\text{ and all }i\in\Z\}. 
$$ 
It is immediate that $Y_{\mathcal{F}}$ is a subshift and we say that $\mathcal{F}$ is a list of {\em forbidden words} in $Y_{\mathcal{F}}$. If $X$ 
is a subshift, define 
$$ 
\mathcal{F}(X):=\{w\in\A^*\colon w\notin\mathcal{L}(X)\}. 
$$ 
It follows immediately from the definitions that $X\subseteq Y_{\mathcal{F}(X)}$ and it follows from the compactness of $X$ that $Y_{\mathcal{F}(X)}\subseteq X$.  Thus $X=Y_{\mathcal{F}(X)}$.  In other words, every subshift can be obtained by specifying an appropriate set of forbidden words.  

On the other hand, if $\mathcal{F}$ is a set of forbidden words and 
$$ 
\CM=\{w\in\mathcal{F}\colon\text{ no proper subword of $w$ lies in $\mathcal{F}$}\} 
$$ 
then $Y_{\CM}=Y_{\mathcal{F}}$.  Thus different sets can present the same subshift through the forbidden words construction.  But there is a canonical ``minimal'' set of forbidden words that presents a subshift.
If we define the {\em minimal forbidden words} $\CM(X)$ by setting 
$$ 
\CM(X):=\{w\in\mathcal{F}(X)\colon\text{ no proper subword of $w$ is in 
$\mathcal{F}(X)$}\} 
$$ 
then $X=X_{\CM(X)}$.  For each $n\in\N$ we define $\CM_n(X):=\CM(X)\cap\A^n$ to be the set of 
minimal forbidden words of length $n$ in $X$. 
The growth rate of the minimal forbidden words was shown to be a conjugacy invariant in B\'eal, Mignosi, and Restivo~\cite{BMRS}, and properties of the subshift that can be deduced from
its presentation via the minimal forbidden words is further studied in~\cite{BMR, MRS}.  

\subsection{Subshifts of finite type and the cover of a subshift} A subshift $X$ is called a {\em subshift of finite type} (SFT) if $\CM(X)$ is finite.  
Bowen~\cite{bowen} gave an equivalent formulation (see~\cite[Theorem 2.1.8]{LM} for a proof): 
\begin{proposition}
\label{prop:bowen}
The shift $(X, \sigma)$ is a subshift of finite type if and only if there 
exists $g\geq 0$ such that whenever $uw, wv\in\CL(X)$ and $|w|\geq g$ then we also have $uwv\in\CL(X)$.  
\end{proposition}

For any subshift $X$ there is a well-known way to write it as the intersection of a descending chain of subshifts of finite type: for each $n\in\N$, define 
$$ 
X_n:=Y_{\bigcup_{k=1}^n\CM_k(X)} 
$$ 
to be the subshift of finite type using the forbidden words up to length $n$ in $X$.  Then 
$$ 
X_1\supseteq X_2\supseteq X_3\supseteq\dots\supseteq X_n\supseteq\dots 
$$ 
and $X=\bigcap_{n=1}^{\infty}X_n$.  This sequence $\{X_n\}_{n=1}^{\infty}$ is called the {\em SFT cover} of $X$.

A subshift $X$ is {\em topologically transitive} if there exists some $x\in X$ such that 
$$ 
X=\overline{\{\sigma^ix\colon i\in\Z\}}. 
$$ 
It is {\em forward transitive} if there exists some $x\in X$ such that 
$$ 
X=\overline{\{\sigma^ix\colon i\in\N\}}. 
$$ 
Parry~\cite{parry} showed that any forward transitive subshift of finite type is {\em intrinsically ergodic}, meaning there is a unique $\sigma$-invariant measure $\mu$, supported on $X$, such that $h_{\mu}(X)=h_{\tp}(X)$ (in other words, the system has a unique measure of maximal entropy).  For general (not necessarily transitive) subshifts of finite type, we record the following elementary lemma: 
\begin{lemma}\label{lem:mmes} 
Every subshift of finite type supports at most finitely many ergodic measures of maximal entropy. 
\end{lemma} 
\begin{proof} 
Let $X$ be an subshift of finite type whose forbidden words all have length at most $n$.  For any $u,v\in\mathcal{L}_n(X)$, 
say that $u\sim v$ if there exist $m_1,m_2>0$ such that $[u]\cap\sigma^{m_1}[v]\neq\emptyset$ and $[v]\cap\sigma^{m_2}[u]\neq\emptyset$.  Let 
$$ 
\mathcal{W}:=\{w\in\mathcal{L}_n(X)\colon w\sim w\}. 
$$ 
Since $X$ is an subshift of finite type whose forbidden words have length 
at most $n$, note that $\sim$ is an equivalence relation on the set $\mathcal{W}$ (it would not necessarily be a transitive relation for more general shifts).  

Let $\mu$ be an ergodic measure of maximal entropy.  For any $u,v\in\mathcal{L}_n(X)$, 
by ergodicity it follows that for $\mu$-almost every $x\in X$ 
$$ 
\lim_{m\to\infty}\frac{1}{m}\sum_{k=0}^{m-1}\one_{[u]}(\sigma^kx)=\mu([u])\quad \text{ and } \quad\lim_{m\to\infty}\frac{1}{m}\sum_{k=0}^{m-1}\one_{[v]}(\sigma^kx)=\mu([v]) 
$$ 
and also 
$$ 
\lim_{m\to\infty}\frac{1}{2m+1}\sum_{k=-m+1}^{m-1}\one_{[v]}(\sigma^kx)=\mu([v]). 
$$ 
It follows from the first two equalities that either $\mu([u])\cdot\mu([v])=0$ or $u\sim v$.  
Using the third equality, it follows that if $\mu([v])>0$ then for $\mu$-almost every $x\in X$ and for every $m\in\Z$ we have $v\sim(x_m,x_{m+1},\dots,x_{m+n-1})$.  There must be at least one $v\in\mathcal{L}_n(X)$ for which $\mu([v])>0$, and for this $v$ we have 
$$ 
\mu\left(\bigcup_{u\sim v}[u]\right)=1. 
$$ 
Furthermore, the measure $\mu$ is supported on the sub-subshift of finite 
type $X(v)$, defined by forbidding 
all words forbidden in $X$ and forbidding all words of length $n$ that are not equivalent to $v$.  
Note that $X(v)$ is a forward transitive subshift of finite type and so by Parry's Theorem~\cite{parry}, $X(v)$ has a unique measure of maximal entropy. 
But $\mu$ is an example of such a measure, since $h_{\mu}(X(v))\leq h_{\tp}(X(v))\leq h_{\tp}(X)=h_{\mu}(X)=h_{\mu}(X(v))$ (since $\mu$ is supported on $X(v)$ and is a measure of maximal entropy on $X$) and so all of the inequalities are actually equalities. 

Thus every ergodic measure of maximal entropy on $X$ is supported on (and 
is a measure of maximal entropy on) a forward transitive subshift of finite type obtained by forbidding words of length $n$ in $X$.  
Since $\mathcal{L}_n(X)$ is finite, there can only be finitely many such forward transitive subshifts of finite type that arise from this construction, 
each of which supports a unique measure of maximal entropy. 
Thus $X$ supports only finitely many measures of maximal entropy. 
\end{proof} 

\section{Previously known criteria implying the existence of characteristic measures} 
\subsection{Finitely many measurably isomorphic systems}
Suppose $(X,\sigma)$ is a subshift and $\mu$ is a $\sigma$-invariant measure supported on $X$.  A key tool that underlies most of the cases where it is known how to prove that characteristic measures exist, is the following: if $\psi\in\Aut(X)$, then $(X,\sigma,\mu)$ and $(X,\sigma,\psi_*\mu)$ are conjugate as measure-preserving systems.  This gives the following 
criterion (which is used implicitly in the literature, for example in~\cite{FT2}) for establishing the existence of characteristic measures: 

\begin{lemma}\label{lem:measurable} 
Let $(X,\sigma)$ be a subshift and suppose there exists a $\sigma$-invariant measure $\mu$ supported on $X$ for which the set  
$$ 
\mathcal{I}(\mu):=\{\nu\colon\text{$\nu$ is a $\sigma$-invariant measure supported on $X$ and $(X,\sigma,\mu)\cong(X,\sigma,\nu)$}\} 
$$ 
is finite.  Then $X$ supports a characteristic measure. 
\end{lemma} 
\begin{proof} 
Define 
$$ 
\xi:=\frac{1}{|\mathcal{I}(\mu)|}\sum_{\nu\in\mathcal{I}(\mu)}\nu. 
$$ 
Then $\xi$ is a $\sigma$-invariant measure supported on $X$.  For any $\psi\in\Aut(X)$, note that $\psi_*$ induces a permutation on $\mathcal{I}(\mu)$ by finiteness of $\mathcal{I}(\mu)$ and the fact that $(X,\sigma,\nu)\cong(X,\sigma,\varphi_*\nu)$ for any $\nu\in\mathcal{I}(\mu)$.  Thus $\psi_*\xi=\xi$. 
\end{proof} 

This gives the existence of a characteristic measure for any full shift, and more generally any shift of finite type: 
\begin{corollary}\label{cor:sft-mme}
Every subshift of finite type has a characteristic measure. 
\end{corollary}
\begin{proof} 
Let $X$ be a subshift of finite type. 
The entropy map, $\mu\mapsto h_{\mu}(\sigma)$, is upper semi-continuous (see~\cite[Theorem 8.2]{W}), and by~\cite[Theorem 8.7, Part (v)]{W}, the 
set of measures of maximal entropy is nonempty, and using Part (iii) of the same theorem, there is an ergodic measure of maximal entropy. 
By Lemma~\ref{lem:mmes}, there are only finitely many ergodic measures of 
maximal entropy.  Since the properties of being ergodic and being a measure of maximal entropy are preserved under measure theoretic isomorphism, Lemma~\ref{lem:measurable} guarantees that $X$ supports a characteristic measure. 
\end{proof} 
The special case that the subshift of finite type is forward transitive is not a new result (we did not 
find the non-transitive case in the literature, but expect Corollary~\ref{cor:sft-mme} to be no surprise to experts).  
Coven and Paul~\cite{CP} 
showed that any automorphism preserves entropy and 
Parry~\cite{parry} showed that a forward transitive subshift of finite type has a unique measure of maximal entropy. 
It follows that this unique measure (of maximal entropy) is preserved under the full automorphism group. However we need further use of some of the tools used to prove Corollary~\ref{cor:sft-mme} as well as its result in our arguments. 

\subsection{Finitely many topologically conjugate systems}
We also make use of a topological analog of Lemma~\ref{lem:measurable}: 
\begin{lemma}\label{lem:topological} 
Let $(X,\sigma)$ be a subshift and suppose there exists a closed subshift 
$Y\subseteq X$ that supports an $\Aut(Y)$-characteristic measure and for which the set 
$$ 
\mathcal{J}(Y):=\{Z\subseteq X\colon(Y,\sigma)\text{ is topologically conjugate to }(Z,\sigma)\} 
$$ 
is finite.  Then $X$ supports a characteristic measure. 
\end{lemma} 
\begin{proof} 
If $Z\in\mathcal{J}(Y)$, then using the  topological conjugacy between $(Y,\sigma)$ and $(Z,\sigma)$  and pushing 
forward the characteristic measure on $\Aut(Y)$, it follows that $Z$ supports an $\Aut(Z)$-characteristic measure 
(in fact $\Aut(Y)\cong\Aut(Z)$ in this case).  
Let $\nu$ be an $\Aut(Y)$-characteristic measure supported on $Y$. 

Every element of $\Aut(X)$ induces a permutation on $\mathcal{J}(Y)$. Thus we can define the group 
$$ 
\Pi:=\{\pi\in\mathrm{Sym}(\mathcal{J}(Y))\colon\psi\text{ induces the permutation $\pi$ for some }\psi\in\Aut(X)\}. 
$$ 
For each $\pi\in\Pi$, choose an automorphism $\phi(\pi)\in\Aut(X)$ that induces the permutation $\pi$ on $\mathcal{J}(Y)$.  Finally define 
$$ 
\xi:=\frac{1}{|\Pi|}\sum_{\pi\in\Pi}\phi(\pi)_*\nu. 
$$ 
If $\psi\in\Aut(X)$ and $\alpha$ is the permutation on $\mathcal{J}(Y)$ induced by $\psi$, 
then $\phi(\alpha)^{-1}\psi$ induces the identity permutation on $\mathcal{J}(Y)$.  Thus $(\phi(\alpha)^{-1}\psi)_*\nu = \nu$, since $\phi(\alpha)^{-1}\psi$ restricted to $Y$ is an automorphism  of $Y$ and $\nu$ is $\Aut(Y)$ characteristic.  More generally, for any permutation $\pi\in\Pi$, we have that $\phi(\alpha)^{-1}\psi$ restricted to $\phi(\pi)(Y)$ is an automorphism of $\phi(\pi)(Y)$ and $\phi(\pi)_*\nu$ is a $\Aut(\phi(\pi)(Y))$ characteristic measure.  Thus, we have: 
\begin{eqnarray*} 
\psi_*\xi &= &\frac{1}{|\Pi|}\sum_{\pi\in\Pi}\psi_*\phi(\pi)_*\nu=
\frac{1}{|\Pi|}\sum_{\pi\in\Pi}\phi(\alpha)_*(\phi(\alpha)^{-1}\psi)_*\phi(\pi)_*\nu \\ 
& =  &
\frac{1}{|\Pi|}\sum_{\pi\in\Pi}\phi(\alpha)_*\phi(\pi)_*\nu 
= \frac{1}{|\Pi|}\sum_{\pi\in\Pi}\phi(\alpha\pi)_*\nu =\frac{1}{|\Pi|}\sum_{\pi\in\Pi}\phi(\pi)_*\nu  = \xi, 
\end{eqnarray*} 
where the 
penultimate equality holds because $\alpha\pi$ runs over every element of $\Pi$ exactly once as $\pi$ runs over $\Pi$.
\end{proof} 
This means, for example, that any subshift that contains periodic points has a characteristic measure, a fact already pointed out in Frisch and Tamuz~\cite{FT2}.  
A version of this fact about periodic points also appears
implicitly in Boyle and Krieger~\cite{BK}, in their defining and study of 
the gyration function.  But by making use of the main result in~\cite{FT2}, this also means that if $X$ has only finitely many minimal subsystems of topological entropy zero, then $X$ has a characteristic measure.

\subsection{Invariant measures for amenable actions}
One final tool for proving the existence of characteristic measures, as a 
classical result in the literature, is 
the Krylov-Bogolioubov Theorem for actions of amenable groups. An immediate application of it yields:
\begin{theorem}[Krylov-Bogolioubov Theorem; see~\cite{KB, bogo, anosov}]
Let $X$ be a subshift and suppose the countable, discrete group $\Aut(X)$ 
is amenable.  Then $X$ supports a characteristic measure. 
\end{theorem}

\section{Language stable shifts}\label{section:main}
\subsection{Defining language stable shifts}
Recall that if $X$ is a subshift and $n\in\N$, then $\CM_n(X)$ denotes the set of minimal forbidden words of length $n$ in $\CL(X)$.  We make the following new definition. 

\begin{definition}\label{def:stable} 
A subshift $(X,\sigma)$ is {\em language stable} if the set 
$$ 
\{n\in\N\colon\CM_n(X)=0\} 
$$ 
has upper Banach density $1$. 
\end{definition} 
Note that any subshift of finite type is language stable, 
because in this case $\CM_n(X)=0$ for all but finitely many $n\in\N$.  
Moreover, any subshift $X$ can be approximated arbitrarily well by unforbidding enough of its forbidden words to make it satisfy Definition~\ref{def:stable} (see Section~\ref{sec:generic} for discussion of the metric on 
subshifts).  More precisely, let $X$ be any subshift and let $S\subseteq\N$ be any set with upper Banach density $1$.  Define 
$$ 
\CM:=\bigcup_{n\notin S}\CM_n(X). 
$$ 
Then $Y_{\CM}$ is language stable, $X$ is a subshift of $Y_{\CM}$, and by 
choosing $S$ to be very sparse we can ensure that $\CM(Y)\setminus\CM(X)$ 
is also very sparse. 

\subsection{Language stable shifts support a characteristic measure}
We use the tools of the last two sections to prove our main theorem: 
\begin{proof}[Proof of Theorem~\ref{thm:main}]
Assume that $X$ be a language stable subshift.  If $X$ is a subshift of finite type, then the result follows from Corollary~\ref{cor:sft-mme}.  Thus it suffices to assume that $X$ is not a subshift of finite type. 

Let $\{X_n\}_{n=1}^{\infty}$ be the SFT cover of $X$.  By Lemma~\ref{lem:mmes},  each $X_n$ supports at most finitely many measures of maximal entropy.  By Lemma~\ref{lem:measurable}, each $X_n$ carries an $\Aut(X_n)$-characteristic measure $\xi_n$.  Thus for each $n\in\N$, we have that $\xi_n$ is a $\sigma$-invariant measure supported on (in general, a proper subset of) $\A^{\Z}$.  By assumption, the set 
$$ 
S:=\{n\in\N\colon\CM_n(X)=0\} 
$$ 
has upper Banach density $1$.  Let $r\colon S\to\N$ be the function 
$$ 
r(s)=-s+\min\{t\notin S\colon t>s\} 
$$ 
defined 
to be the longest run of consecutive integers that lie in $S$, starting from $s$ (since $X$ is not an subshift of finite type, this run is well-defined).  
By assumption, the set $S$ has upper Banach density $1$, and so we can choose a subset 
$S^{\prime}\subseteq S$ along which $r(s)$ is strictly increasing. 

Since $\A^{\Z}$ is a compact metric space, it follows from the Banach-Alaoglu Theorem that the set $\{\xi_{s^{\prime}+r(s^{\prime})-1}\}_{s\in S^{\prime}}$ has weak* accumulation points in the set of all Borel measures on $\A^{\Z}$.  Let $\xi$ be one such accumulation point.  We claim that $\xi$ is a characteristic measure for $X$.  Passing from $S^{\prime}$ to $S^{\prime\prime}$ if necessary, we can assume the weak* limit of the sequence $\{\xi_{s^{\prime}+r(s^{\prime})-1}\}_{s^{\prime}\in S^{\prime}}$ exists. 

Let $\psi\in\Aut(X)$ be fixed.  By the Curtis-Hedlund-Lyndon Theorem (Theorem~\ref{th:CHL}), $\psi$ is a block code of some range $R\geq 0$.  Let $\Psi\colon\mathcal{L}_{2R+1}(X)\to\A$ be such that for all $x\in X$ and all $i\in\Z$ we have 
$$ 
(\psi x)_i=\Psi(x_{i-R},\dots,x_i,\dots,x_{i+R}). 
$$ 
If $R^{\prime}\geq R$, 
then $\psi$ is also a block code of range $R^{\prime}$ and  so we may assume (increasing the value of $R$ if necessary) that $R$ is the range for both $\psi$ and $\psi^{-1}$.  From hereon, we fix such $R$.  
By abusing notation, we now extend the domain of $\psi$ to include any element of $\A^{\Z}$ to which the range $R$ block code defining $\psi$ can be applied.  Since the function $r(s)$ is strictly increasing on $S^{\prime}$, there exists some $M$ such that $r(s^{\prime})\geq2R+1$ for all $s^{\prime}\in S^{\prime}$ satisfying $s^{\prime}>M$; let some such $s^{\prime}$ be fixed.  Notice that 
\begin{equation}\label{eq:string} 
X_{s^{\prime}}=X_{s^{\prime}+1}=\dots=X_{s^{\prime}+r(s^{\prime})-1} 
\end{equation}  
since these are all subshifts of finite type with identical sets of minimal forbidden words.  
Furthermore, since $X\subseteq X_{s^{\prime}+r(s^{\prime})-1}$ and there is no word of length at most $s^{\prime}+r(s^{\prime})-1$  forbidden in $X$ that was not also forbidden in $X_{s^{\prime}+r(s^{\prime})-1}$, it follows that 
$$ 
\mathcal{L}_{s^{\prime}+r(s^{\prime})-1}(X)=\mathcal{L}_{s^{\prime}+r(s^{\prime})-1}(X_{s^{\prime}+r(s^{\prime})-1}). 
$$ 
For each $w\in\mathcal{L}_{s^{\prime}+r(s^{\prime})-1}(X)$, applying the sliding block code $\Psi$ determines a word $\Psi(w)\in\mathcal{L}_{s^{\prime}+r(s^{\prime})-1-2R}(X)$.  This implies that 
$$ 
\mathcal{L}_{s^{\prime}+r(s^{\prime})-1-2R}(\psi(X_{s^{\prime}+r(s^{\prime})-1}))\subseteq\mathcal{L}_{s^{\prime}+r(s^{\prime})-1-2R}(X_{s^{\prime}+r(s^{\prime})-1-2R}) 
$$ 
and so $\psi(X_{s^{\prime}+r(s^{\prime})-1})\subseteq X_{s^{\prime}+r(s^{\prime})-1-2R}$.  
Since $r(s^{\prime})\geq2R+1$, it follows from this and~\eqref{eq:string} 
that
 $\psi(X_{s^{\prime}+r(s^{\prime})-1})\subseteq X_{s^{\prime}+r(s^{\prime})-1}$. 
 
 Applying the analogous argument with $\psi^{-1}$ instead of $\psi$ 
 shows that $\psi^{-1}(X_{s^{\prime}+r(s^{\prime})-1})\subseteq X_{s^{\prime}+r(s^{\prime})-1}$, and so 
$$ 
\psi(X_{s^{\prime}+r(s^{\prime})-1})=X_{s^{\prime}+r(s^{\prime})-1}. 
$$ 
In other words, $\psi\in\Aut(X_{s^{\prime}+r(s^{\prime})-1})$.  But by definition, $\xi_{s^{\prime}+r(s^{\prime})-1}$ is a $\Aut(X_{s^{\prime}+r(s^{\prime})-1})$-characteristic measure and so 
\begin{equation}\label{eq:char} 
\psi_*\xi_{s^{\prime}+r(s^{\prime})-1}=\xi_{s^{\prime}+r(s^{\prime})-1}. 
\end{equation} 
Since~\eqref{eq:char} holds for any $s^{\prime}\in S^{\prime}$ such that $s^{\prime}>M$, it follows that  
$$ 
\psi_*\xi=\psi_*\left(\lim_{s^{\prime}\in S^{\prime}}\xi_{s^{\prime}+r(s^{\prime})-1}\right)=\lim_{s^{\prime}\in S^{\prime}}\psi_*\xi_{s^{\prime}+r(s^{\prime})-1}=\lim_{s^{\prime}\in S^{\prime}}\xi_{s^{\prime}+r(s^{\prime})-1}=\xi. 
$$ 
Since this holds for any $\psi\in\Aut(X)$, we have that $\xi$ is a characteristic measure for $X$. 
\end{proof} 

Note that in the proof, the measure $\xi$ produced is a weak-* limit of the measures $\xi_n$ that are 
measures of maximal entropy on the shifts of finite type $X_n$.  For each 
$n\in\N$ we have $h_{\tp}(X_n)\geq h_{\tp}(X)$ (since $X\subseteq X_n$) and so $h_{\xi_n}(\sigma)\geq h_{\tp}(X)$.  By upper semi-continuity of the entropy map (see~\cite[Theorem 8.2]{W}), 
it follows that $h_{\xi}(\sigma)\geq h_{\tp}(X)$.  Since $\xi$ is supported on $X$ we have equality, so it is a measure of maximal entropy on $X$. 
 Therefore our proof shows that: 
\begin{corollary}
\label{cor:language-stable}
Any language stable shift has a characteristic measure that is a measure of maximal entropy. 
\end{corollary}

\section{Genericity of language stable shifts}
\label{sec:generic}
We show that the set of language stable shifts is a dense $G_{\delta}$, with respect to the Hausdorff topology, in both the space of all subshifts 
with a fixed alphabet and in the subspace of positive entropy subshifts with a fixed alphabet.

We start by defining the distance between two subshifts over the same alphabet:  
if $\mathcal{A}$ is a finite alphabet and $X, Y\subseteq\mathcal{A}^{\Z}$ 
are two subshifts, define 
\begin{equation}
\label{def:distance}
d(X,Y):=2^{-\inf\{n\colon\CL_n(X)\neq\CL_n(Y)\}}. 
\end{equation} 
This metric gives the usual Hausdorff metric on the space of subshifts of $\A^{\Z}$.  Endowed with the metric~\eqref{def:distance}, the space of all subshifts of $\mathcal{A}^{\Z}$ is 
a compact metric space (recall that by definition, a subshift $X$ is a closed, $\sigma$-invariant subset of $\mathcal A^\Z$). 

Fixing a finite alphabet $\A$, a property of subshifts  is said to be {\em generic} if it holds for a dense $G_{\delta}$ subset of the space of all subshifts, in this topology.  A property is {\em generic among shifts of positive entropy} if it defines a dense $G_{\delta}$ (in the induced topology) in the subspace of subshifts of positive entropy.  
Let $S$ denote the set of all subshifts of $\mathcal{A}^{\Z}$ and 
let $S^+$ denote the set of all subshifts of $\A^{\Z}$ with positive topological entropy. 
For any $c> 0$, let $S^+_{\geq c}$ denote the set of all subshifts of $\A^\Z$ with topological entropy greater than or equal to $c$.  

Frisch and Tamuz~\cite{FT2} show that subshifts with zero entropy are generic in the space of all shifts, and that more generally subshifts with  entropy $c$ are generic in the space of all shifts with entropy at least $c$.  Along these lines, we prove that thee language stable subshifts are 
generic. We start by showing that these shifts are a $G_\delta$ subset: 
\begin{theorem}\label{th:stable-g-delta} 
The set of language stable subshifts in $S$ is a $G_{\delta}$ subset of $S$. 
\end{theorem}
\begin{proof} 
A subshift $X$ is language stable if and only if for all $k\in\N$ there exists $n_k\in\N$ such that 
$$ 
X_{n_k}=X_{n_k+1}=X_{n_k+2}=\dots=X_{n_k+k-1}, 
$$ 
where $\{X_n\}_{n=1}^{\infty}$ denotes the SFT cover of $X$, and  we  show that the set of subshifts in $S$ having this property is $G_{\delta}$. 

For each fixed $n\in\N$, note that 
$$ 
\mathcal{W}_n:=\{\mathcal{L}_n(X)\colon X\in S\} 
$$ 
is finite, as $\mathcal{L}_n(X)\subseteq\mathcal{L}_n(\A^{\Z})$ for any $X\in S$.  Enumerate the elements of $\mathcal{W}_n$ as $L_1^n,L_2^n,\dots,L_{|\mathcal{W}_n|}^n$.  For each $1\leq i\leq|\mathcal{W}_n|$, let $\mathcal{X}(i,n)$ denote the subshift of finite type whose set of forbidden words is $\mathcal{L}_n(\A^{\Z})\setminus L_i^n$.  For each $k$, let 
$$ 
\mathcal{B}(i,n,k):=\{Y\in S\colon d(\mathcal{X}(i,n),Y)<2^{-n-k+1}\}. 
$$ 
In other words, $Y\in\mathcal{B}(i,n,k)$ if and only if $\mathcal{L}_{n+k-1}(Y)=\mathcal{L}_{n+k-1}(\mathcal{X}(i,n))$.  For any such $Y$, it follows that $\mathcal{L}_{j}(Y)=\mathcal{L}_{j}(\mathcal{X}(i,n))$ for any $1\leq j\leq n+k-1$.  Fix such a $Y$ and let $\{Y_m\}_{m=1}^{\infty}$ be the SFT cover of $Y$.  Then, since $\mathcal{X}(i,n)$ is a subshift of finite type whose forbidden words all have length at most $n$, we have 

$$ 
Y_n=Y_{n+1}=Y_{n+2}=\dots=Y_{n+k-1}. 
$$ 
Since $\mathcal{B}(i,n,k)$ is open, the set 
$$ 
\mathcal{U}_k:=\bigcup_{n=1}^{\infty}\bigcup_{i=1}^{|\mathcal{W}_n|}\mathcal{B}(i,n,k) 
$$ 
is open.  Therefore the set 
$$ 
\LS :=\bigcap_{k=1}^{\infty}\mathcal{U}_k 
$$ 
is $G_{\delta}$.  

We claim $\LS$ is precisely the set of language stable shifts.  A subshift $Y\in \LS$ if and only if for all $k\geq 1$ there exists $n$ and $1\leq 
i\leq|\mathcal{W}_n|$ such that $Y\in\mathcal{B}(i,n,k)$.  We have already seen that the statement $Y\in\mathcal{B}(i,n,k)$ implies that 
$$ 
Y_n=Y_{n+1}=\dots=Y_{n+k-1}, 
$$ 
where $\{Y_n\}_{n=1}^{\infty}$ is the SFT cover of $Y$.  So if $Y\in \LS$ then for all $k$ there exists some $n_k$ such that 
$$ 
Y_{n_k}=Y_{n_k+1}=\dots=Y_{n_k+k-1}, 
$$ 
and this is equivalent to being language stable.  Thus $\LS$ is contained 
in the set of language stable shifts.  Conversely, for any language stable shift $Y$ and any $k\geq 1$, there exists $n_k$ such that 
$$ 
Y_{n_k}=Y_{n_k+1}=\dots=Y_{n_k+k-1}. 
$$ 
Let $i_k\in\{1,\dots,|\mathcal{W}_{n_k}|\}$ be the index for which $\mathcal{L}_{n_k}(Y)=L_i^{n_k}$.  Then $Y\in\mathcal{B}(i_k,n_k,k)\subseteq\mathcal{U}_k$.  Since this holds for any $k$, $Y\in\bigcap_k\mathcal{U}_k$, meaning $Y\in \LS$.  This proves the claim that the the $G_{\delta}$ set $\LS$ is equal to the set of language stable shifts. 
\end{proof} 

\begin{corollary}
The set of language stable subshifts in $S^+$ is a $G_{\delta}$ subset of 
$S^+$.  
More generally, for any $c \geq 0$, the set of language stable subshifts in $S^+_c$ is a $G_{\delta}$ subset of $S^+_c$. 
\end{corollary}
\begin{proof} 
By definition of the induced topology, the intersection of $S^+$ (respectively $S^+_{\geq c}$) with any $G_{\delta}$ subset of $S$ is a $G_{\delta}$ subset of $S^+$ (respectively $S^+_{\geq c}$), so both the statements follow from Theorem~\ref{th:stable-g-delta}. 
\end{proof} 

\begin{theorem}\label{th:stable-dense} 
For any $c\geq 0$, the set of language stable subshifts in $S_{\geq c}^+$ 
is dense in $S_{\geq c}^+$. 
Analogously, the set of language stable subshifts in $S$ is dense in $S$ and the set of language stable subshifts in $S^+$ is dense in $S^+$.
\end{theorem} 
\begin{proof} 
Let $\mathcal{U}$ be an arbitrary nonempty, open subset of $S_{\geq c}^+$ and fix some $X\in\mathcal{U}$.  Let $\{X_n\}_{n=1}^{\infty}$ be the SFT cover of $X$.  Then, by definition of the metric, there exists $N$ such that 
$$ 
\{Y\in S^+\colon\mathcal{L}_N(Y)=\mathcal{L}_N(X)\}\subseteq\mathcal{U}. 
$$ 
For such an $N$ notice that $X_N$ is language stable (since it is a shift of finite type), that $\mathcal{L}_N(X)=\mathcal{L}_N(X_N)$, and that 
\begin{equation}
\label{eq:entropy-ineq}
h_{\tp}(X_N)\geq h_{\tp}(X)\geq c, 
\end{equation} since $X\subseteq X_N$.  
Therefore $X_N\in\mathcal{U}$.  Since $\mathcal{U}$ was arbitrary, the set of language stable subshifts in $S_{\geq c}^+$ is dense in $S_{\geq c}^+$ for any $c\geq 0$.

For the analogous results for $S$ and $S^+$, the only small modification is that in inequality~\eqref{eq:entropy-ineq} 
we only have $h_{\tp}(X)\geq0$ in $S$ and $h_{\tp}(X)>0$ in $S^+$.  
\end{proof} 

Combining the results of this section, we get: 

\begin{corollary}
The set of language stable subshifts of $S$ is generic in $S$, the set of 
language stable subshifts of $S^+$ is generic in $S^+$, and for any $c\geq0$ the set of language stable subshifts of $S_{\geq c}^+$ is generic in $S_{\geq c}^+$. 
\end{corollary} 

\section{The example}
\label{sec:example}
\subsection{Language stability gives rise to new shifts with characteristic measures}
In this section, we construct a language stable subshift that carries a characteristic measure that cannot be seen to exist for any of the four reasons given in the introduction.  In addition, we show that this characteristic measure is a measure of maximal entropy.
Specifically, we  build a language stable subshift $(W, \sigma)$ with the 
following properties: 
	\begin{enumerate} 
	\item The automorphism group of the subshift $W$ is nonamenable. 
	\item The subshift $W$ and each of its nonempty subsystems has positive topological entropy. 
	\item Each characteristic measure supported on $W$ is measurably isomorphic to infinitely many other measures also supported on $W$. 
	\item Each closed, proper subshift $W'\subset W$ either is topologically conjugate to infinitely many other subshifts 
	of $W$ (meaning Lemma~\ref{lem:topological} does not apply to it), or has strictly lower topological entropy than $W$ (meaning that even if Lemma~\ref{lem:topological} could be applied to it, the resulting measure would not 
	be a measure of maximal entropy on $W$, and therefore would not be the measure obtained from Corollary~\ref{cor:language-stable}). 
		\end{enumerate} 
In summary, $W$ carries a characteristic measure that is a measure of maximal entropy (by Corollary~\ref{cor:language-stable}) and the existence of this measure is not implied by any of the four previously known methods of finding a characteristic measure.
	
Our method is to build the shift $W$  as a product $W = X\times Y\times 
Z$ of three subshifts, and the properties of two of these subshifts may be of independent interest.  The shift $X$, described in Section~\ref{sec:X}, is a language stable, positive entropy, minimal subshift.  
The shift $Y$ is the full shift on $2$ symbols and its properties are described in Section~\ref{sec:Y}. 
The shift $Z$, described in Section~\ref{sec:Z},  has countably many ergodic measures, all of which are  isomorphic to each other, and in Section~\ref{sec:example-summary} we check that the product system $W$ is language stable and has the properties described above. 
In the constructions of each of these component systems, we need to take care of how they interact with each other, both to guarantee that the resulting system is language stable and to ensure that the existence of a characteristic measure for the constructed system is not covered by any of the previously known methods. 

Before giving our example, we reiterate that $W$ is language stable and so has a characteristic measure by Corollary~\ref{cor:language-stable}.  One might wonder if there is an easier way to see this and so we indicate why what may seem like a natural approach to finding a characteristic measure on $W$ is not viable without significant extra information. 
 It is tempting to try to find a characteristic measure for $W$ by analyzing $X$, $Y$, and $Z$ separately, finding an $\Aut(X)$-characteric measure $\alpha$, an $\Aut(Y)$-characteristic measure $\beta$, and an $\Aut(Z)$-characteristic measure $\gamma$, and guessing that $\alpha\times\beta\times\gamma$ may be an $\Aut(W)$-characteristic measure.  
However, a recent result of Salo and Schraudner~\cite{salo-schraudner} shows that the automorphism group of a product of systems may be much larger than the product of their individual automorphism groups.  Specifically, they show that if $X\subseteq\{0,1\}^{\Z}$ is the {\em sunny side up shift}, 
consisting of all configurations with at most one $1$, then $\Aut(X)\cong\Z$ (in fact it is generated by the shift). 
On the other hand, they showed that $\Aut(X\times X)\cong(\Z^{\infty}\rtimes S_{\infty})\rtimes(\Z^2\rtimes S_2)$.  
Thus, returning to our example, although $\alpha\times\beta\times\gamma$ is certainly invariant under the subgroup $\Aut(X)\times\Aut(Y)\times\Aut(Z)\subseteq\Aut(X\times Y\times Z)$, there is no reason to believe it is 
$\Aut(W)$-invariant without fully describing the algebraic structure of $\Aut(W)$, which may be significantly larger.

\subsection{There exists a language stable, positive entropy, minimal subshift}
\label{sec:X}

\begin{lemma}\label{lem:build}
Let $S$ be a forward transitive subshift of finite type whose minimal forbidden words all have length at most $N$ and let $\varepsilon>0$.  For any $m\in\N$, and any sufficiently large $d\in\N$ we have
\begin{equation}
\label{eq:shiftS}
S^{\prime}:=\{x\in S\colon\text{every subword of $x$ of length $d$ contains every word in $\mathcal{L}_m(S)$}\}
\end{equation}
is a forward transitive subshift of finite type and $h_{\tp}(S^{\prime})\geq h_{\tp}(S)-\varepsilon$.
\end{lemma}

\begin{proof}
If $S$ contains only a single periodic point, then the lemma is trivially 
true.  Thus without loss,  we can assume that $S$ contains at least two distinct periodic points.  

For a fixed $m\in\N$, $S^{\prime}$ is the subshift of finite type obtained by forbidding all the minimal forbidden words in $S$, as well as all words in $\mathcal{L}(S)$ of length $d$ that omit at least one element of $\mathcal{L}_m(S)$.  
Since $S$ is a subshift of finite type, by Proposition~\ref{prop:bowen} there exists $g\in\N$ such that if $u,v,w\in\mathcal{L}(S)$ are such that $uv,vw\in\mathcal{L}(S)$ and $|v|\geq g$, then $uvw\in\mathcal{L}(S)$.  
For the remainder of the proof, we assume that $m\geq g$.  

If $d\geq m$  is fixed and the shift
 $S^{\prime}$ is defined by~\eqref{eq:shiftS}, then this shift is nonempty for sufficiently large $d$.
Fixing such $d$, for any $a,b\in\mathcal{L}(S^{\prime})$ we can
find words $a^{\prime},b^{\prime}\in\mathcal{L}(S^{\prime})$ such that:
$|a^{\prime}|\geq|a|+d$, $|b^{\prime}|\geq|b|+d$, $a$ is the leftmost subword of $a^{\prime}$, $b$ is the
rightmost subword of $b^{\prime}$, and the rightmost subword of length $m$ in $a^{\prime}$ coincides with the leftmost subword of length $m$ in $b^{\prime}$.  Note that we can do this because all words in $\mathcal{L}(S^{\prime})$
can be legally extended arbitrarily far to the right and left (in at least one way) and all words of length $d$ in $S^{\prime}$ contain a copy of every word in $\mathcal{L}_m(S)$.  Therefore,
setting $b^{\prime\prime}$ to be the word obtained by omitting the leftmost $m\leq d$ letters of $b^{\prime}$, we
have that $a^{\prime}b^{\prime\prime}\in\mathcal{L}(S^{\prime})$ and $a^{\prime}b^{\prime\prime}=acb$ for some $c\in\mathcal{L}(S^{\prime})$.  Thus for any two words $a,b\in\mathcal{L}(S^{\prime})$, there exists $c\in\mathcal{L}(S^{\prime})$ such that $acb\in\mathcal{L}(S^{\prime})$ and so 
$S^{\prime}$ is forward transitive.

We are left with showing that given $\varepsilon> 0$, we can choose $d\in\N$ sufficiently large such that
$h_{\tp}(S^{\prime})\geq h_{\tp}(S)-\varepsilon$.  We carry this out in two stages. 

\subsubsection*{Step 1: we build a high entropy shift missing only one word} 
Let $y$ be a periodic point in $S$ of minimal period $p$ and choose $w\in\mathcal{L}_p(X)$ such that $y=\dots wwww\dots$.  Since $S$ is a forward transitive shift of finite type that contains at least two distinct periodic points, we can find a word $q$, whose length is a multiple of $|w|$, that neither begins nor ends with the word $w$, and\ is such that $wqw\in\mathcal{L}(S)$.  Since $|w|\geq m$ it follows that if $u,v\in\mathcal{L}(S)$ are such that both $uw,wv\in\mathcal{L}(S)$, then by Proposition~\ref{prop:bowen} we have $uwqwv\in\mathcal{L}(S)$.

Let $\mu$ be the Parry measure on $S$ and note that $\mu$ is an ergodic and nonatomic measure satisfying $h_{\mu}(\sigma)=h_{\tp}(S)$ (see~\cite{parry}). 
By the Shannon-McMillan-Breiman Theorem, for $\mu$-almost every $z\in S$ we have
$$
\lim_{n\to\infty}-\frac{1}{2n+1}\log\mu([z_{-n}\dots z_0\dots z_n])=h_{\mu}(\sigma)=h_{\tp}(S).
$$
Thus for any $\delta>0$, there exists $N\in\N$ and a set $Q\subseteq S$ such that $\mu(Q)>1-\delta/2$ and such that 
$$
\left|-\frac{1}{2n+1}\log\mu([z_{-n}\dots z_0\dots z_n])-h_{\mu}(\sigma)\right|<\delta
$$
for all $z\in Q$ and all $n\geq N$.  For any fixed $k>\frac{|q|}{|w|}+3$ (recall that the words $q$ and $w$ are defined in the beginning of this step), 
let $w(k):=\underbrace{ww\dots w}_{\text{$k$ times}}$.  By the Pointwise Ergodic Theorem, we have
$$
\lim_{n\to\infty}\frac{1}{2n+1}\sum_{i=-n}^n1_{[w(k)]}(\sigma^iz)=\mu\left([w(k)]\right)
$$
for $\mu$-almost every $z\in S$.  Therefore there exists $M\in\N$ and a set $R\subseteq S$ such that $\mu(R)>1-\delta/2$ and 
$$
\left|\frac{1}{2n+1}\sum_{i=-n}^{n-|w|-1}1_{[w(k)]}(\sigma^iz)-\mu\left([w(k)]\right)\right|<\delta
$$ 
for all $z\in R$ and all $n\geq M$.  Taking the maximum of $N$ and $M$, without loss of generality we can assume that $N=M$.  Then $\mu(Q\cap R)>1-\delta$ and 
for any $z\in Q\cap R$, we have that for all $n\geq N$, 
\begin{align}
\label{eq:upper-z}
\exp\bigl(-(2n+1)\cdot(h_{\tp}(S)+\delta)\bigr)  & <\mu([z_{-n}\dots z_0\dots z_n])  \\ 
&
\nonumber  < \exp\bigl(-(2n+1)\cdot(h_{\tp}(S)-\delta)\bigr)
\end{align}
and 
\begin{equation*}
\mu([w(k)])-\delta<\frac{1}{2n+1}\sum_{i=-n}^{n-|w|-1}1_{[w(k)]}(\sigma^iz)<\mu([w(k)])+\delta. 
\end{equation*}
Let $\mathcal{W}(n,k)\subseteq\mathcal{L}_{2n+1}(S)$ denote 
the set of words of the form $z_{-n}\dots z_0\dots z_n$ for some $z\in Q\cap R$ and fix some $n\geq N$.  
Since $\mu(Q\cap R)>1-\delta$, using the upper bound for $\mu([z_{-n}\dots z_0\dots z_n])$ 
given in~\eqref{eq:upper-z}, for each $z\in Q\cap R$ there are at least
$$
\frac{1-\delta}{\exp\bigl(-(2n+1)\cdot(h_{\tp}(S)-\delta)\bigr)}
$$
cylinder sets based on words of length $2n+1$ in $\mathcal{L}(S)$ for which the number of occurrences of $w(k)$ as a subword of them lies between $(2n+1)(\mu([w(k)])-\delta)$ and $(2n+1)(\mu([w(k)])+\delta)$.

Recall that the word $w(k)$ is a periodic self-concatenation of the word $w$ and recall the definition of the word $q$ from the first paragraph of Step 1.  Define the modified word $\tilde{w}(k)$ to be the word 
$$
\tilde{w}(k):=wwq\hspace{-0.125 in}\underbrace{ww\cdots w}_{k-\frac{|q|}{|w|}-2\text{ times}}
$$
which
is just $w(k)$ with each occurrence of $w$ starting with the third  through the $(|q|/|w|+2)^{nd}$ occurrence  changed to the word $q$.  
 Note that in any element of $S$ and any occurrence of $w(k)$ in that element, the modified sequence that replaces this occurrence of $w(k)$ with $\tilde{w}(k)$ also 
determines an element of $S$, by choice of $q$.  Moreover, making this replacement cannot introduce new occurrences of the word $w(k)$, by the Fine-Wilf Theorem~\cite{FW} and the fact that $q$ neither begins nor ends with $w$, because $\tilde{w}(k)$ begins and ends with two consecutive occurrences of $w$.  Therefore, for any fixed $k$ and any element of $S$ we can produce an element of $S$ that does not contain $w(k)$ as a subword, by enumerating all occurrence of $w(k)$, inductively modifying the next occurrence that remains at each stage $\tilde{w}(k)$, and taking a limit of the sequence of elements of $S$ obtained at each stage. 
Let $Y_k\subseteq S$ be the subshift obtained by forbidding the word $w(k)$.  Notice that if $u\in\mathcal{L}(S)$ does not contains $w(k)$ as a subword, then either $u\in\mathcal{L}(Y_k)$ or every element of $S\cap[u]$ contains an occurrence of $w(k)$.  The latter is only possible if all elements of $S\cap[u]$ have an occurrence of $w(k)$ that overlaps the word $u$ in the location determined by the cylinder set (otherwise modify occurrences of $w(k)$ to $\tilde{w}(k)$ as described previously).  Such a word can be modified at most once on its right and once on its left so that these occurrences of $w(k)$ are instead the word $\tilde{w}(k)$.  Now, any element of $\mathcal{L}(S)$ can be turned  
 into an element fo $\mathcal{L}(S)$ that does not contain $w(k)$ as a subword using the procedure described above (enumerating occurrences of $w(k)$ that occur within it beginning from the left and inductively modifying the next that remains be $\tilde{w}(k)$ until we reach the end of the word).  Therefore for any $n$ there is a map from $\mathcal{L}_{2n+1}(S)$ to $\mathcal{L}_{2n+1}(Y_k)$ that is at most $(2^{2+(2n+1)(\mu([\tilde{w}(k)])+\delta)})$-to-one when restricted to the words arising from the restriction of elements of $Q\cap R$ to the set $\{-n,\dots,0,\dots,n\}$ (first removing all occurrences of $w(k)$ and then making at most once change on the right and on the left if all elements of $S\cap[u]$ contain an occurrence of $w(k)$ that intersects $u$).  For fixed $\delta>0$ and all sufficiently large $n$, this map is at most $(2^{(2n+1)(\mu([\tilde{w}(k)])+2\delta)})$-to-one.

We claim that $h_{\tp}(Y_k)$ can be made arbitrarily close to $h_{\tp}(S)$ by taking $k$ sufficiently large.  To see this note that for any $n\geq N$, the number of words in $\mathcal{L}_{2n+1}(Y_k)$ is at least
$$
\frac{\left(\frac{1-\delta}{\exp(-(2n+1)\cdot(h_{\tp}(S)-\delta))}\right)}{2^{(2n+1)(\mu([\tilde{w}(k)])+2\delta)}}
$$
since the map described above taking a word in $\mathcal{L}(S)$ to a word in $\mathcal{L}(Y_k)$ 
is at most $(2^{(2n+1)\mu([\tilde{w}(k)])+2\delta})$-to-one on the words arising from the restriction of elements of $Q\cap R$ to the set $\{-n,\dots,0,\dots,n\}$.  By taking $k$ sufficiently large we can guarantee 
that $\mu([w(k-|q|/|w|)])$ is arbitrarily close to zero, and so $\mu([\tilde{w}(k)])$ is also (since it contains $w(k-|q|/|w|)$ as its rightmost subword).  Therefore the exponential growth rate of $|\mathcal{L}_{2n+1}(Y_k)|$ is at least $h_{\tp}(S)-3\delta$ for any sufficiently large $k$.

\subsubsection*{Step 2: we use $Y_k$ to define $S^{\prime}$} 
Choose a word $w\in\mathcal{L}(S)$ that contains every element of $\mathcal{L}_m(S)$ as a subword.  
For each $k\geq 1$, we use this word $w$ in the construction of $Y_k$.  
Fixing $\varepsilon>0$, we can choose $k$ sufficiently large  such that $h_{\tp}(Y_k)>h_{\tp}(S)-\varepsilon/4$.  
Choose $N$ such that for all $n\geq N$, we have $\log P_X(n)>n(h_{\tp}(S)-\varepsilon/2)$. 
 Since $Y_k$ is a forward transitive shift of finite type, by Proposition~\ref{prop:bowen} 
there exists a bound $B$ such that for any $v\in\mathcal{L}(Y_k)$ there are 
words $g_{w,v},h_{w,v}\in\mathcal{L}(Y_k)$ of length $|g_{w,v}|,|h_{w,v}|\leq B$ and such that $wg_{w,v}vh_{w,v}w\in\mathcal{L}(Y_k)$.  
 Choose one such $g_{w,v}$ and $h_{w,v}$ for each $v\in\mathcal{L}(Y_k)$.  Then there exist $b_1,b_2\leq B$ such that 
$$ 
\lim_{n\to\infty}\frac{1}{n}\log\left|\left\{v\in\mathcal{L}_n(Y_k)\colon|g_{w,v}|=b_1\text{ and }|h_{w,v}|=b_2\right\}\right|>h_{\tp}(S)-\varepsilon/4. 
$$
Define 
$$ 
\mathcal{W}_n:=\left\{wg_{w,v}vh_{w,v}\in\mathcal{L}_n(Y_k)\colon|g_{w,v}|=b_1\text{ and }|h_{w,v}|=b_2\right\}. 
$$ 
Without loss of generality, increasing $N$ if necessary, we can assume that $\left|\mathcal{W}_n\right|>n(h_{\tp}(S)-\varepsilon/2)$ for all $n\geq N$.  By Proposition~\ref{prop:bowen}, the elements of $\mathcal{W}_n$ can be freely concatenated. 

For each fixed $n\in\N$, note that the shift $S^{\prime}$ defined using 
$d:=|w|+b_1+b_2+n$ contains all elements of $S$ that 
can be written as concatenations of elements of $\mathcal{W}_n$. 
Define the shift $Z_n$ to be the shift by only allowing elements of $S$ that can be written as bi-infinite concatenations of elements of $\mathcal{W}_n$.  
Then the topological entropy of  $S^{\prime}$ 
is at least as large as the topological entropy of the shift $Z_n$.  We are left with estimating the entropy of $Z_n$.  

For $\ell,n\geq N$,
 the number of words in $\mathcal{L}_{\ell(n+b_1+b_2+|w|)}(Z_n)$ is at least $|\mathcal{W}_n|^\ell>(h_{\tp}(S)-\varepsilon/2)^{\ell n}$; namely, all words 
in $\mathcal{W}_n$ have the same length and therefore, all ways of concatenating $\ell$ elements of $\mathcal{W}_n$ result in distinct words of length $\ell(n+b_1+b_2+|w|)$ in $\mathcal{L}(S)$. 
For any $\delta>0$, we can take $n$ sufficiently large such 
that $\ell(n+b_1+b_2+|w|)<\ell n(1+\delta)$ and thus 
$$P_{Z_n}(\ell(n+b_1+b_2+|w|))>(h_{\tp}(S)-\varepsilon/2)^{\ell(n+b_1+b_2+|w|)/(1+\delta)}.
$$
 For $\delta$ sufficiently small, this is larger than $(h_{\tp}(S)-\varepsilon)^{\ell(n+b_1+b_2+|w|)}$. 
Fixing some sufficiently large $n$ such that this holds, we have that this estimate holds for all $\ell\geq N$ and so 
$$ 
\liminf_{\ell\to\infty}\frac{1}{\ell(n+b_1+b_2+|w|)}\log P_{Z_n}(\ell(n+b_1+b_2+|w|))>h_{\tp}(S)-\varepsilon. 
$$ 
But 
$$ 
\lim_{t\to\infty}\frac{1}{t}\log P_{Z_n}(t)=h_{\tp}(Z_n) 
$$ 
and, in particular, the limit exists. Thus we have $h_{\tp}(S^{\prime})\geq h_{\tp}(Z_n)>h_{\tp}(S)-\varepsilon$. 
\end{proof}

For $\alpha\in[0,1)$, let $R_\alpha\colon [0,1)\to [0,1)$ denote the rotation $x\mapsto x+\alpha\pmod 1$. 

\begin{lemma}\label{lemma:twoparts} 
Let $S$ be a forward transitive subshift of finite type. 
 Let $\alpha\in\mathbb{R}\setminus\mathbb{Q}$,
let $k\in\N$, let $\beta_1,\beta_2,\dots,\beta_k\in(0,1)$, and 
let $Z_i$ be the shift obtained by coding the circle rotation $([0,1),R_{\alpha})$ with respect to the partition $\{[0,\beta_i),[\beta_i,1)\}$.  Fix  $n_1\in\N$ and $u\in\mathcal{L}_{n_1}(S)$.  
\begin{enumerate}
\item \label{lem:semi-minimal}
  For any sufficiently large 
$n_2\in\N$, there is a word $v\in\mathcal{L}_{n_2}(S)$ such that for any $1\leq i\leq k$ and words $w_1\in\mathcal{L}_{n_1}(Z_i)$ and $w_2\in\mathcal{L}_{n_2}(Z_i)$, the word $u\times w_1$ occurs as a subword of $v\times w_2$. 
\item \label{lem:build-new}
For any sufficiently large $d$, the shift 
$$ 
S^{\prime}:=\{x\in S\colon\text{every subword of $x$ of length $d$ contains every word in $\mathcal{L}_{n_2}(S)$}\}
$$ 
is a forward transitive subshift of finite type with $h_{\tp}(S^{\prime})\geq h_{\tp}(S)-\varepsilon$ and 
such that for any $1\leq i\leq k$, all words in 
$\mathcal{L}_{n_1}(S^{\prime}\times Z_i)$ occur syndetically, with gap at 
most $2d$, in every element of $S^{\prime}\times Z_i$.
\end{enumerate}
\end{lemma}

\begin{proof}
To prove Part~\eqref{lem:semi-minimal}, fix $u\in\mathcal{L}_{n_1}(S)$.  Since $S$ is forward transitive, there exists a word $y\in\mathcal{L}(S)$ such that 
$$\tilde{y} = \cdots uyuyuyuy\cdots$$
is a periodic point in $X$.  Let $m=|u|+|y|$ be the period of $\tilde{y}$

Fix some $1\leq i\leq k$ and some word $w_1\in\mathcal{L}_{n_1}(Z_i)$.  The coding of a point $x\in[0,1)$ begins with the word $w_1$ if and only if $x$ lies in the cell of the partition 
$$ 
\bigvee_{i=0}^{n_1-1}
R_{\alpha}^{-i}\{[0,\beta_i),[\beta_i,1)\} 
$$
that corresponds to the word $w_1$.  This cell is a half-open interval of 
positive length.  Since $([0,1),R_{\alpha}^m)=([0,1),R_{m\alpha})$ is minimal, there exists $t(\beta_i)>0$ such that for any $x\in[0,1)$ there is some $0\leq n\leq t(\beta_i)$ for which $R_{m\alpha}^n(x)=R_{\alpha}^{mn}(x)$ is in this half-open interval.   Moreover, notice that $t(\beta_i)$ is bounded above by a function that depends only on the length of the shortest interval in 
$$\bigvee_{i=0}^{n_1-1}
R_{\alpha}^{-i}\{[0,\beta_i),[\beta_i,1)\}$$ (namely the time it takes for orbits under $R_{m\alpha}$ to become more dense than the length of the shortest interval).  Set $t:=\max\{t(b_j)\colon 1\leq j\leq k\}$. 

Set $n_2:=mt+|u|$ with $m=|u|+|y|$ to be the period of $\tilde{y}$ chosen and $t = t(\beta_i)$.  
Let $w_2\in\mathcal{L}_{n_2}(Z_i)$.  Find some $x\in[0,1)$ such that the coding of the orbit of $x$ with respect to the partition $\{[0,\beta_i),[\beta_i,1)\}$ begins with the word $w_2$.  By the definition of $t$, there exists some $0\leq n\leq t$ such that the word $w_1$ occurs as a subword of $w_2$, beginning exactly $mn$ letters from the left of $w_2$.  Let $v\in\mathcal{L}_{n_2}(S)$ be the word $uyuyuy\cdots yu$ that has length $n_2$.  Then the subword of $v$ that begins exactly $mn$ letters from the left of $v$ is $u$.  In particular, the word $u\times w_1$ occurs as a subword of $v\times 
w_2$, starting $mn$ letters from the left of $v\times w_2$.  This completes the proof of Part~\eqref{lem:semi-minimal}. 

We turn to Part~\eqref{lem:build-new}.  The statement that $S^{\prime}$ is a forward transitive subshift of 
finite type with $h_{\tp}(S^{\prime})\geq h_{\tp}(S)-\varepsilon$ for any 
$d$ sufficiently large follows immediately from Lemma~\ref{lem:build}.  Fix $1\leq i\leq k$.  
Let $u\times w_1\in\mathcal{L}_{n_1}(S^{\prime}\times Z_i)$, where $u\in\mathcal{L}_{n_1}(S^{\prime})$ and $w_1\in\mathcal{L}_{n_1}(Z_i)$.  
Let $v\in\mathcal{L}_{n_2}(S^{\prime})$ be the word $v$ constructed in Part~\eqref{lem:semi-minimal}.  Let $w_2\in\mathcal{L}_{n_2}(Z_i)$.  Then $v\times w_2\in\mathcal{L}_{n_2}(S^{\prime}\times Z_i)$ and so  by definition of $S^{\prime}$ this word 
occurs in every length $d$ subword of every element of $S^{\prime}$.  But $u\times w_1$ occurs as a 
subword of $v\times w_2$, and 
so $u\times w_1$ occurs in every length $d$ subword of every element of $S^{\prime}\times Z_i$.  Since $u\times w_1\in\mathcal{L}_{n_1}(S^{\prime}\times Z_i)$ is arbitrary, this holds for all such words.
\end{proof}

\begin{lemma}\label{lem:facts} 
Let $\alpha\in\mathbb{R}\setminus\mathbb{Q}$ and let $0<\beta_1<\beta_2<1$.  For $i=1,2$,  let $Z_i$ be the subshift obtained by coding the system $([0,1),R_{\alpha})$ by the partition $\{[0,\beta_i),[\beta_i,1)\}$.  Then $Z_1$ and $Z_2$ are both minimal, uniquely ergodic, and there exists $N\in\N$ such that for all $n\geq N$ we have $\mathcal{L}_n(Z_1)\cap\mathcal{L}_n(Z_2)=\emptyset$.
\end{lemma}
\begin{proof}
The irrational circle rotation $([0,1),R_{\alpha})$ is uniquely ergodic and its unique invariant measure is Lebesgue measure.  Fix $i\in\{1,2\}$ and define the partition $\mathcal{P}_n:=\bigvee_{j=0}^nR_{\alpha}^{-j}\{[0,\beta_i),[\beta_i,1)\}$.  
The  cells in the partition $\mathcal{P}_n$ determine (distinct) elements of $\mathcal{L}_{n+1}(Z_i)$ and every word in $\mathcal{L}_{n+1}(Z_i)$ corresponds to the coding, according to the partition $\mathcal{P}_0$, of the elements in a cell of $\mathcal{P}_n$.  The cells of $\mathcal{P}_n$ are half-open subintervals of $[0,1)$ and so by unique ergodicity of $([0,1),R_{\alpha})$, for any cell $\mathcal{C}\in\mathcal{P}_n$ and any $\varepsilon>0$ there exists $M$ such that for all $m\geq M$ and all $x\in[0,1)$ we have 
$$ 
\left|\lambda(\mathcal{C})-\sum_{i=0}^m1_{\mathcal{C}}\left(\mathcal{P}_n(R_{\alpha}^ix)\right)\right|<\varepsilon,  
$$ 
where $\mathcal{P}_n(R_{\alpha}^ix)$ is the cell of $\mathcal{P}_n$ that contains $R_{\alpha}^ix$.  In particular, if $\mu_i$ is the push-forward of $\lambda$ under the coding map, then for any $m\geq M$ the frequency with which $w\in\mathcal{L}_{n+1}(Z_i)$ occurs as a subword of any $u\in\mathcal{L}_m(Z_i)$ differs from $\mu_i([w])$ by at most $\varepsilon$.  Since this $M=M(n)$ exists for any $n$, $Z_i$ is uniquely ergodic.

To see that $Z_i$ is minimal, fix any $w\in\mathcal{L}(Z_i)$.  Then $w$ corresponds to the coding, under $\mathcal{P}_0$, of the points in one of the cells of $\mathcal{P}_n$ for some $n$.  The cell of $\mathcal{P}_n$ that corresponds to $w$ is a half-open interval and the minimal system $([0,1),R_{\alpha})$ has the property that the orbit visits this half-open interval syndetically with uniform gap between consecutive visits.  
This means that all sufficiently long words in $\mathcal{L}(Z_i)$ contain $w$ syndetically as a subword, with uniform gap between consecutive occurrences.  So $Z_i$ is minimal. 

Finally, set $\varepsilon:=(\beta_2-\beta_1)/2$ and by unique ergodicity find $M\in\N$ such that for $i=1,2$, any $x\in Z_i$, and any $m\geq M$ we have 
$$ 
\left|\mu_i([0])-\sum_{i=0}^m1_{[0]}(\sigma^ix)\right|<\varepsilon. 
$$ 
This means that for any $m\geq M$ the number of times that $0\in\mathcal{L}_1(Z_1)$ occurs as a subword of any element of $\mathcal{L}_{m+1}(Z_1)$ is strictly smaller than the number of times $0\in\mathcal{L}_1(Z_2)$ occurs as a subword of any element of $\mathcal{L}_{m+1}(Z_2)$.  
Thus  $\mathcal{L}_{m+1}(Z_1)\cap\mathcal{L}_{m+1}(Z_2)=\emptyset$ for all $m\geq M$. 
\end{proof}

\begin{lemma}\label{lem:low-entropy} 
Let $\alpha\in\mathbb{R}\setminus\mathbb{Q}$ and let $\beta\in(0,1)$ be such that $\beta=n\alpha\pmod1$ for some integer $n>0$.  Let $Z\subseteq\{0,1\}^{\Z}$ be the subshift obtained by coding the circle rotation $([0,1),R_{\alpha})$ with respect to the partition $\{[0,\beta),[\beta,1)\}$.  For any integer $k\geq1$, let $S\colon\{0,1,\dots,k-1\}\to\{0,1,\dots,k-1\}$ be the map $S(i):=i+1\pmod{k}$.  There exists a subshift $Y\subseteq\{0,1\}^{\Z}$ that is topologically conjugate to $(Z\times\{0,1,\dots,k-1\},\sigma\times S)$.  
\end{lemma}
\begin{proof}
Let $Z_{\alpha}\subseteq\{0,1\}^{\Z}$ be the coding of $([0,1),R_{\alpha})$ with respect to the partition $\{[0,\alpha),[\alpha,1)\}$, meaning that 
$Z_{\alpha}$ is the Sturmian shift with rotation angle $\alpha$.  By a theorem of Durand~\cite[Corollary 12]{D}, $Z_{\alpha}$ is Cantor prime, meaning that any nontrivial system that $Z_{\alpha}$ factors  onto (in the topological sense) 
is topologically conjugate to $Z_{\alpha}$.  

If $\gamma=m\alpha\pmod1$ for some integer $m>0$, then the intervals $[0,\gamma)$ and $[\gamma,1)$ can be written as unions of the cells of the partition $\bigvee_{i=-m}^mR_{\alpha}^i(\{[0,\alpha),[\alpha,1)\}$, and so $Z_{\alpha}$ factors onto the (infinite) shift obtained by coding $([0,1),R_{\alpha})$ with respect to the partition $\{[0,\gamma),[\gamma,1)\}$.  For ease of notation, call this shift $Z_{\gamma}$.  
By Durand's Theorem,  there exists a topological conjugacy 
$\varphi_{\gamma}\colon Z_{\alpha}\to Z_{\gamma}$.  
Since the system $Z_\alpha$ is Sturmian, it has a unique asymptotic pair,  
meaning there is a unique pair $x[\alpha],y[\alpha]\in Z_{\alpha}$ such that $x[\alpha]_i=y[\alpha]_i$ for all $i>0$ but $x[\alpha]_0\neq y[\alpha]_0$.  Moreover, this pair has the property that $x[\alpha]_i=y[\alpha]_i$ for all $i<-1$ and $\{(x[\alpha]_{-1},x[\alpha]_0),(y[\alpha]_{-1},y[\alpha]_0)\}=\{(1,0),(0,1)\}$.  
Similarly, since $\gamma=m\alpha\pmod1$, there is a unique pair $x[\gamma],y[\gamma]\in Z_{\gamma}$ such that $x[\gamma]_i=y[\gamma]_i$ for all $i>0$ but $x[\gamma]_0\neq y[\gamma]_0$.  
For this system, we have $x[\gamma]_i=y[\gamma]_i$ for all $i<-m-1$ and $\{(x[\gamma]_{-m-1},x[\gamma]_0),
(y[\gamma]_{-m-1},y[\gamma]_0)\}=\{(1,0),(0,1)\}$.  These points correspond to the limit of sequences of codings of points that approach $\gamma$ from the left and from the right in $[0,1)$ with respect to the partition $\{[0,\gamma),[\gamma,1)\}$, and the fact that $0$ is the $k^{th}$ preimage of $\gamma$ under $R_{\alpha}$.  
Note that the conjugacy $\varphi_{\gamma}$ sends the set $\{x[\alpha],y[\alpha]\}$ to the set $\{\sigma^tx[\gamma],\sigma^ty[\gamma]\}$ for some $t\in\Z$.  Without loss of generality, passing from $\varphi_{\gamma}$ to $\varphi_{\gamma}\circ\sigma^{-t}$ if necessary, we can assume that $t=0$.  
In particular, if $z\in\{x[\alpha],y[\alpha]\}$, then $\varphi_{\gamma}(z)_0$ is determined by $\varphi_{\gamma}(z)_{-m-1}$, and vice-versa.

For $0\leq s<k$, define $\gamma_s:=k(s+1)\alpha\pmod1$.  Define the 
map $\psi\colon Z_{\alpha}\times\{0,1,\dots,k-1\}\to\{0,1\}^{\Z}$ by the formula 
$$ 
\left(\psi(x,i)\right)_j:=\left(\varphi_{\gamma_{i+j}}(x)\right)_j
$$ 
where the subscript $i+j$ of $\gamma_{i+j}$ is understood modulo $k$.  The map $\psi$ is given by a block code and so is continuous and commutes with the map $(x,i)\mapsto(\sigma(x),i+1)$, where again addition in the second coordinate 
is taken modulo $k$.  
We claim  that $\psi$ is injective. 
To prove the claim, we proceed by contradiction and let $(x_1,i_1)$ and $(x_2,i_2)$ be distinct points for which $\psi(x_1,i_1)=\psi(x_2,i_2)$.  If $i_1\neq i_2$, then the restriction of $x_1$ to the set $\{kn\colon n\in\Z\}$ is the same as the restriction of $x_2$ to this set.  Therefore there is a coding of the system $([0,1),R_{k\alpha})$ with respect to the partition $\{[0,\gamma_{i_1}),[\gamma_{i_1},1)\}$ that coincides with a coding of $([0,1),R_{k\alpha})$ with respect to the partition $\{[0,\gamma_{i_2}),[\gamma_{i_2},1)\}$.  This is impossible since, by Lemma~\ref{lem:facts} the languages of these symbolic systems are disjoint for all sufficiently large size words.  Therefore $i_1=i_2$ and so we can assume that $x_1\neq x_2$.  

For any fixed $0\leq d<k$, the restriction of $x_1$ to $\{d+nk\colon n\in\Z\}$ coincides with the restriction of $x_2$ to this set.  Recall that $Z_\alpha$ can be written as the union of 
the orbits $\mathcal{O}(x[\alpha])$ and $\mathcal{O}(y[\alpha])$ and all codings of points in $[0,1)\setminus\mathcal{O}(\alpha)$ with respect to $\{[0,\alpha),[\alpha,1)\}$. 
If $x\in Z_{\alpha}$ is the coding of some point $y\in[0,1)\setminus\mathcal{O}(\alpha)$, 
we claim that no other element of $Z_{\alpha}$ has the same restriction as $x$ has 
to the set $\{kn\colon n\in\Z\}$.  To see this, note that this restriction can be interpreted as the coding of the point $y$ in the system $([0,1),R_{k\alpha})$ with 
respect to the partition $\{[0,\alpha),[\alpha,1)\}$.  
Therefore $y$ is determined by the restriction of $x$ to $\{kn\colon n\in\Z\}$ and since $y\notin\{0,\alpha\}$, $x$ is determined by $y$.  Therefore $x_1,x_2\in\mathcal{O}(x[\alpha])\cup\mathcal{O}(y[\alpha])$ and 
since they have the same restriction to $\{kn\colon n\in\Z\}$ there must exist some $t$ for which $x_1=\sigma^t(x[\alpha])$ and $x_2=\sigma^t(y[\alpha])$.  In particular, the only locations where $x_1$ and $x_2$ differ from each other are $t$ and $t+1$, where one has the symbols $(1,0)$ and the other has the symbols $(0,1)$.  Let $0\leq d<k$ be such that $t\equiv d\pmod k$.  Since for any $s$, $\varphi_{\gamma_s}$ sends the set $\{x[\alpha],y[\alpha]\}$ to the set $\{x[\gamma_s],y[\gamma_s]\}$, it follows that $\varphi_{\gamma_d}(x_1)$ and $\varphi_{\gamma_d}(x_2)$ do not coincide on the set $\{d+kn\colon n\in\Z\}$, and we have a contradiction.  Therefore $x_1=x_2$ and so the map $\psi\colon Z_{\alpha}\times\{0,1,\dots,k-1\}\to\{0,1\}^{\Z}$ is injective.  

Since $\psi$ is injective and $Z\times\{0,1,\dots,k-1\}$ is a compact metric space, $\psi$ is a homeomorphism on its image. 
Thus $\psi$ defines a topological conjugacy on its image, meaning that the image of $Y:=Z_{\alpha}\times\{0,1,\dots,k-1\}$ under $\psi$ is a subshift of $\{0,1\}^{\Z}$ such that $(Y,\sigma)$ is topologically conjugate to $(Z_{\alpha}\times\{0,1,\dots,k-1\},\sigma\times S)$.  But $Z_{\alpha}$ is topologically conjugate to $Z$ and so $Y$ is also topologically conjugate to $(Z\times\{0,1,\dots,k-1\},\sigma\times S)$.
\end{proof}

\begin{lemma}
\label{lemma:def-X}
There exists a language stable, positive entropy, minimal subshift.  Moreover, given a sequence of intervals of consecutive integers
$$ 
\{i_1,i_1+1,\dots,i_1+j_1\}, \{i_2,i_2+1,\dots,i_2+j_2\},\dots,\{i_n,i_n+1,\dots,i_n+j_n\},\dots 
$$ 
 such that both sequences $i_\ell$ and $j_\ell$ are increasing 
there exists a language stable, positive entropy, minimal subshift $X$ which has the property that $X_{i_n}=X_{i_n+j_n}$ for infinitely many $n$, where $X_k$ denotes the $k^{th}$ term in the SFT cover of $X$.  Furthermore, 
if $\alpha\in\mathbb{R}\setminus\mathbb{Q}$, $\beta_1,\beta_2,\beta_3,\ldots\in(0,1)$, and $Z_i$ is the coding of the circle rotation $([0,1),R_{\alpha})$ with respect to the partition $\{[0,\beta_i),[\beta_i,1)\}$, then for each $i\in\N$ the shift $X\times Z_i$ is also minimal.
\end{lemma}
\begin{proof}
We inductively construct a sequence of subshifts 
$$ 
X_1\supseteq X_2\supseteq\dots\supseteq X_n\supseteq X_{n+1}\supseteq\cdots 
$$ 
and show that $X:=\bigcap_{i=1}^{\infty}X_i$ is the desired system. 

Fix a sequence $\{\varepsilon_i\}_{i=1}^{\infty}$ of elements of $(0,1)$ such that $\sum_i\varepsilon_i<\log 2$.  Let $X_1:=\{0,1\}^{\Z}$ and so $h_{\tp}(X_1)=\log 2$.  Fix the parameter $n_1:=1$.  By Part~\eqref{lem:build-new} of Lemma~\ref{lemma:twoparts},  
we can choose $d_1>n_1$ such that 
$$ 
X_2:=\{x\in X_1\colon\text{every subword of length $d_1$ contains every 
word in $\mathcal{L}_{n_1}(X_1)$}\} 
$$ 
is a forward transitive subshift of finite type satisfying $h_\tp(X_2)\geq h_\tp(X_1)-\varepsilon_1$, and such that all words in $\mathcal{L}_1(X_2\times Z_1)$ occur syndetically with gap at most $2d_1$.  

Assume that we have inductively constructed a nested sequence of topologically transitive subshifts of finite type $X_1\supseteq X_2\supseteq\dots\supseteq X_m$, as well as parameters $n_i, d_i$ for all $i=1,\dots,m-1$ satisfying 
$$ 
n_{i+1}\geq\max\{d_i+i,i_k+j_k\}
$$ 
where $i_k=\min\{i_\ell\colon i_\ell>d_i\}$ and $1\leq i<m-1$.  Suppose 
further that for all $x\in X_m$ and all $1\leq i<m$, every subword of $x$ 
of length $d_i$ contains every element of $\mathcal{L}_{n_i}(X_i)$ as a subword.  Finally suppose 
that every word in $\mathcal{L}_{i-1}(X_i\times Z_j)$ 
occurs syndetically with gap at most $2d_i$ for all $1\leq j<i$.  
Set $k^{\prime}=\min\{i_\ell\colon i_\ell>d_{m-1}\}$,  define $n_m:=\max\{d_{m-1}+m,i_{k^{\prime}}+j_{k^{\prime}}\}$.
Again using Part~\eqref{lem:build-new} of Lemma~\ref{lemma:twoparts}, we can choose $d_m>n_m$ such that 
$$ 
X_{m+1}:=\{x\in X_m\colon\text{every subword of length $d_m$ contains every word in $\mathcal{L}_{n_m}(X_m)$}\} 
$$ 
is a forward transitive subshift of finite type satisfying $h_\tp(X_{m+1})\geq h_\tp(X_m)-\varepsilon_m$ and, increasing $d_m$ if necessary, such that every word in $\mathcal{L}_m(X_{m+1}\times Z_j)$ occurs syndetically with gap at most $2d_m$ for all $1\leq j<m+1$.  Inductively, this defines the shift $X_i$ for all $i\geq1$ and we set $X:=\bigcap_iX_i$.  We claim that $X$ is minimal, language stable, and has positive entropy, and  further claim that $X\times 
Z_j$ is minimal for all $j\in\N$.

We first check that $X$ is minimal.  Let $w\in\mathcal{L}(X)$ and pick $i\in\N$ such 
that $n_i>|w|$.  Since $\mathcal{L}_{n_i}(X)\subseteq\mathcal{L}_{n_i}(X_i)$, it follows by construction that for any $x\in X_{i+1}$ the word $w$ occurs in every subword of $x$ of length $d_i$.  But $X\subseteq X_{i+1}$ 
and so $w$ occurs syndetically, with gap at most $d_i$, in every element of $X$.  Since this holds for any $w\in\mathcal{L}(X)$, $X$ is minimal. 

Next we check that for any $j\in\N$, the shift $X\times Z_j$ is minimal.  Let $w\times w^{\prime}\in\mathcal{L}(X\times Z_j)$.  By construction, there exists $I$ such that $w\in\mathcal{L}(X_i)$ for all $i\geq I$, and without 
loss we can assume $I>j$.  By construction, every word in $\mathcal{L}_{I+|w|}(X_{I+|w|+1}\times Z_j)$ occurs syndetically with gap at most $2d_{I+|w|}$.  In particular, $w\times w^{\prime}$ occurs syndetically with at most this gap.  
But $X\times Z_j\subseteq X_{I+|w|+1}\times Z_j$ and so $w\times w^{\prime}$ occurs syndetically, with gap at most $2d_{I+|w|}$ in every element of $X\times Z_j$.  Since this holds for any $w\times w^{\prime}\in\mathcal{L}(X\times Z_j)$, $X\times Z_j$ is minimal.

Next we show that $X$ is language stable.  Fix $i\in\N$ and recall that $n_{i+1}\geq\max\{d_i+i,i_k+j_k\}$ where $i_k=\min\{i_\ell\colon i_\ell>d_i\}$.  The shift $X_{i+1}$ is a forward transitive subshift of finite type whose minimal forbidden words all have length at most $d_i$.  By construction, every word in $\mathcal{L}_{n_{i+1}}(X_{i+1})$ occurs in every element of $X_{i+2}$ and hence in every element of $X$.  Therefore $\mathcal{L}_{n_{i+1}}(X)=\mathcal{L}_{n_{i+1}}(X_{i+1})$.  It follows that $\mathcal{L}_k(X)=\mathcal{L}_k(X_{i+1})$ for all $1\leq k\leq n_{i+1}$.  
Since there are no minimal forbidden words in $X_{i+1}$ of length greater 
than $d_i$ and since $n_{i+1}\geq \max\{d_i+i,i_k+j_k\}$ where $i_k=\min\{i_\ell\colon i_\ell>d_i\}$, it follows that there are no minimal forbidden words in $X$ of lengths $d_i+1,\dots,d_i+i$, 
and moreover  there are no forbidden words of any length lying in the interval $\{i_k,i_k+1,\dots,i_k+j_k\}$.  Since this holds for any $i\in\N$, it follows that the set 
$$ 
\{k\colon\text{$X$ has no minimal forbidden words of length $k$}\} 
$$ 
has upper Banach density $1$.  In other words, $X$ is language stable.  
Furthermore, there are infinitely many $n$ for which $X_{i_n}=X_{i_n+j_n}$. 

Finally we show that $h_\tp(X)>0$.  Since $X_i$ is a topologically transitive subshift of finite type, it follows from Parry~\cite{parry} that $X_i$ supports a unique measure of maximal entropy $\mu_i$.  Passing to a subsequence if necessary, we can assume that  the sequence $\{\mu_i\}_{i=1}^{\infty}$ converges to a weak* limit $\mu$.  Note that $\mu$ is supported on $X=\bigcap_iX_i$ and by upper semi-continuity of the entropy map (see~\cite[Theorem 8.2]{W}) for subshifts, we have that 
$$ 
h_{\mu}(\sigma)\geq\limsup_{i\to\infty}h_{\mu_i}(\sigma)\geq h_\tp(X_1)-\sum_{i=1}^{\infty}\varepsilon_i=\log(2)-\sum_{i=1}^{\infty}\varepsilon_i>0. 
$$ 
By the Variational Principle (see for example~\cite[Theorem 8.6]{W}) we have $h_\tp(X)\geq h_{\mu}(\sigma)>0$. 
\end{proof} 

\subsection{The characteristic measures on a full shift}
\label{sec:Y}

Boyle, Lind, and Rudolph~\cite[Corollary 10.2]{BLR} show that for any topologically mixing subshift of finite type, the measure of maximal entropy 
is the unique characteristic measure of positive entropy, and all characteristic measures of entropy zero are countable convex combinations of purely atomic measures supported on unions of periodic orbits.  In particular, for the full shift on two symbols, their result says the following: 
\begin{lemma}\label{lem:full-shift}
For the full $2$-shift $(\{0,1\}^\Z , \sigma)$, 
every characteristic measure is a convex combination of the 
Bernoulli measure (assigning measure $(1/2)^n$ to each cylinder set determined by a word of length $n$) and atomic 
measures supported on unions of periodic orbits.
\end{lemma}

\subsection{There exists a language stable shift with countably many ergodic measures, all of whose ergodic measures are  isomorphic to each other} 
\label{sec:Z}

\begin{lemma}\label{lem:lang-setup}
Let $\alpha,\beta\in(0,1)$ and $\alpha\notin\mathbb{Q}$.  Let $R_{\alpha}\colon[0,1)\to[0,1)$ be the map $R_{\alpha}(x):=x+\alpha\pmod1$.  Let $\P:=\{[0,\beta),[\beta,1)\}$ and let $X_{\beta}$ be the coding of the system $([0,1),R_{\alpha})$ by the partition $\P$.  For $n\geq2$, let 
$$ 
\mathcal{P}_n:=\bigvee_{i=1}^{n-1} R_{\alpha}^{-i}\P. 
$$ 
Finally let $\mathcal{S}_0=\{0,\beta\}$ and $\mathcal{S}_n:=R_{\alpha}^{-n}\mathcal{S}_0=\{-n\alpha,\beta-n\alpha\}$.  
If $X_{\beta}$ has a minimal forbidden word of length $n+1$, then at least one element of $\R_0$ lies in the same cell of $\mathcal{P}_n$ as an element of $\mathcal{S}_{n+1}$.
\end{lemma}
\begin{proof}
Suppose $X_{\beta}$ has a minimal forbidden word $w$ of length $n+1$.  Writing $w:=(a_0,a_1,\dots,a_{n-1},a_n)\in\A^{n+1}$, 
then since $w$ is minimal, we have that $$(a_0,a_1,\dots,a_{n-1}), (a_1,a_2,\dots,a_n)\in\mathcal{L}_n(X_{\beta}).$$  
Let $u:=(a_1,a_2,\dots,a_{n-1})\in\mathcal{L}_{n-1}(X_{\beta})$ denote the ``interior'' of $w$.  

Say there is a unique $b\in\A$ such that $ub\in\mathcal{L}_n(X_{\beta})$. 

Then since $(a_1,a_2,\dots,a_n)\in\mathcal{L}_n(X_{\beta})$, 
it follows that $b=a_n$.  But then since 
$a_0u=(a_0,a_1,\dots,a_{n-1})\in\mathcal{L}_n(X_{\beta})$, 
it follows  that $w=a_0ua_n\in\mathcal{L}_{n+1}(X_{\beta})$, as 
$$ 
\{x\in X_{\beta}\colon x_i=a_i\text{ for all }0\leq i<n\}\neq\emptyset, 
$$ 
and so the only possibility is that $x_n=a_n$ because $(x_1,\dots,x_{n-1})=u$.  
But this contradicts the assumption that $w$ is a forbidden word, and so we conclude that 
there is no unique $b\in\A$ such that $ub\in\mathcal{L}_n(X_{\beta})$.  Similarly there cannot be a unique $c\in\A$ such that $cu\in\mathcal{L}_n(X_{\beta})$.  

Recall that when $([0,1),R_{\alpha})$ is coded by the partition $\P$, the 
cells of $\P_n$ are in one-to-one correspondence with cylinder sets of the form: 
$$ 
\sigma^{-1}[v]_0^+:=\{x\in X_{\beta}\colon x_i=v_i\text{ for all }1\leq i\leq n-1\}
$$ 
where $v\in\mathcal{L}_{n-1}(X_{\beta})$.  
We maintain the same notation for $w$ being a minimal forbidden word of length $n+1$ and $u$ denoting its interior.  
Since there is no unique $b\in\A$ such that $ub\in\mathcal{L}(X_{\beta})$, 
this means that the cell of $\P_n$ that corresponds to $u$ is subdivided into at least two different cells in the refined partition $\bigvee_{i=1}^n R_{\alpha}^{-i}\P=\P_n\vee R_{\alpha}^{-n}\P$.  The only two cells 
that are subdivided in this way are the cells containing the elements of $\R_n$.  
Similarly, since there is no unique $c\in\A$ such that $cu\in\mathcal{L}(X_{\beta})$, it follows that the cell of $\P_n$  corresponding to $u$ is subdivided into at least two different cells in the refined partition $\bigvee_{i=0}^{n-1} R_{\alpha}^{-i}\P=\P\vee\P_n$.  The only two cells that are subdivided in this way are the cells containing the elements of $\R_0$.  Therefore, the cell of $\P_n$ corresponding to $u$ contains at least one element of $\R_0$ and at least one element of $\R_n$.
\end{proof} 

In preparation for our next lemma, we define a second partition $\Q:=\{[0,\alpha),[\alpha,1)\}$ and for $n>2$, define $\Q_n:=\bigvee_{i=2}^{n-1}R_{\alpha}^{-i}\Q$ (note that this partition does not include $\Q\vee R_\alpha^{-1}\Q$).
  We make use of an auxiliary result. 
\begin{lemma}\label{lem:P-Q}
For fixed $n>2$, every cell of $\Q_n$ can be written as a union of cells from $\P_n$. 
\end{lemma} 
\begin{proof}
Observe that $\Q_n$ is the partition of $[0,1)$ into intervals whose endpoints come from the set 
$$ 
\{-\alpha,-2\alpha,-3\alpha,\dots,-(n-1)\alpha\}. 
$$ 
Notice that $\P_n$ is the partition of $[0,1)$ into intervals whose endpoints come from the set 
$$ 
\{-\alpha,-2\alpha,3-\alpha,\dots,-(n-1)\alpha\}\cup\{\beta-\alpha,\beta-2\alpha,\beta-3\alpha,\dots,\beta-(n-1)\alpha\}. 
$$ 
Therefore each interval in $\Q_n$ can be written as a union of intervals from $\P_n$.
\end{proof} 

The intervals that comprise $\Q_n$ have a natural adjacency relation in $\mathbb{R}/\Z$: we say two cells of $\Q_n$ are {\em adjacent} if they share an endpoint in $\mathbb{R}/\Z$ and are {\em twice adjacent} if there is a third cell that is adjacent to both of them.

\begin{lemma}\label{lem:lang-reduction} 
Maintaining the notation of Lemma~\ref{lem:lang-setup}, if $X_{\beta}$ has a minimal forbidden word of length $n+1$ then at least one of the following holds: 
	\begin{enumerate}
	\item $\phantom{-}0$ and $-n\alpha$ lie in the same, adjacent, or twice adjacent cells of $\Q_n$; \label{case1}
	\item $\phantom{-}\beta$ and $-n\alpha$ lie in the same cell of $\Q_n$; \label{case2}
	\item $-\beta$ and $-n\alpha$ lie in the same, adjacent, or twice adjacent cells of $\Q_n$. \label{case3}
	\end{enumerate} 
\end{lemma} 
\begin{proof} 
By Lemma~\ref{lem:lang-setup}, at least one element of $\{0,\beta\}$ must lie in the same cell of $\P_n$ as an element of $\{-n\alpha,\beta-n\alpha\}$.  If $-n\alpha$ lies in the same cell of $\P_n$ as an element of $\{0,\beta\}$, then one of~\eqref{case1} or~\eqref{case2} occurs since the cells of $\Q_n$ can be written as unions of cells in $\P_n$.  Otherwise, one of the following holds: 
	\begin{enumerate}
	\item $\beta$ and $\beta-n\alpha$ lie in the same cell of $\P_n$ (hence also the same cell of $\Q_n$); 
	\item $0$ and $\beta-n\alpha$ lie in the same cell of $\P_n$ (hence also 
the same cell of $\Q_n$). 
	\end{enumerate} 
Let $d$ denote the metric on $[0,1)$ inherited from the Euclidean metric on $\mathbb{R}/\Z$.  Note that $d(\beta,\beta-n\alpha)=d(0,-n\alpha)$, and so if $\beta$ and $\beta-n\alpha$ lie 
in the same cell of $\P_n$ then $d(0,-n\alpha)$ is at most the length, $L$, of that cell.  By the Three Lengths Theorem~\cite{sos}, for any $n>2$ the intervals comprising $\Q_n$ have at least two and at most three distinct lengths.  Moreover, for any $n$ where the intervals have three distinct lengths, the longest of the lengths is the sum of the shorter two and the sum of the lengths of any two consecutive cells is at least the longest length.  Therefore, since $d(0,-n\alpha)=L$ is at most the longest length of any interval in $\Q_n$, $0$ and $-n\alpha$ can be in the same cell of $\Q_n$, adjacent cells of $\Q_n$, or twice adjacent cells of $\Q_n$.  
In particular, if $\beta$ and $\beta-n\alpha$ lie in the same cell of $\P_n$ then~\eqref{case1} holds.  Finally, if $0$ and $\beta-n\alpha$ lie in 
the same cell of $\P_n$, then $d(-\beta,-n\alpha)=d(0,\beta-n\alpha)$ and similarly~\eqref{case3} holds.
\end{proof} 

The interest in Lemma~\ref{lem:lang-reduction} is that we have removed the dependence on $\P_n$ (a partition which depends on $\beta$) and replaced it with $\Q_n$ (a partition which depends only on $\alpha$). 

We recall some facts that follow from the Three Lengths Theorem of Sos~\cite{sos}.  First, if $n$ is such that $\Q_n$ has only two distinct lengths, then a new length is created in $\Q_{n+1}$ by subdividing one of the intervals from $\Q_n$ with the longest length.  Further, there is a simple 
formula for the number of intervals of each length in $\Q_n$: 
\begin{theorem}[{See for example~\cite[Theorem 2.6.1]{AS}}]
Let $\alpha\notin\mathbb{Q}$ and let $\alpha=[0,a_1,a_2,\dots]$ be its continued fraction expansion.  Let $p_k/q_k=[0,a_1,a_2,\dots,a_k]$ be its $k^{th}$ convergent.  Then the sequence $\{q_k\}_{k=1}^{\infty}$ 
is nondecreasing, tends to infinity, and for every integer $n\geq1$ there 
exists a unique $k$ such that there are numbers $1\leq m\leq a_{k+1}$ and 
$0\leq r<q_k$ satisfying $n=mq_k+q_{k-1}+r$.  The partition $\Q_n\vee \Q\vee R_\alpha^{-1}\Q$ has:
	\begin{enumerate} 
	\item $r+1$ intervals of length $\eta_{k-1}-m\eta_k$ (Type 1); 
	\item $n+1-q_k$ intervals of length $\eta_k$ (Type 2); 
	\item $q_k-r-1$ intervals of length $\eta_{k-1}-(m-1)\eta_k$ (Type 3), 
	\end{enumerate} 
where $\eta_k:=(-1)^k(q_k\alpha-p_k)$. 
\end{theorem} 
We note that the partition $\Q_n\vee \Q\vee R_\alpha^{-1}\Q$  in the statement of this theorem is the standard partition used in a continued fraction approximation. 

An immediate corollary is the following: 
\begin{corollary}\label{cor:classical-old} 
No interval in $\Q_{q_k+q_{k-1}}\vee \Q\vee R_\alpha^{-1}\Q$ 
is divided into more than $2$ subintervals in $\Q_{(a_{k+1}+1)q_k+q_{k-1}-1}\vee \Q\vee R_\alpha^{-1}\Q$.  In particular, the orbit segment 
$$ 
\{-n\alpha\colon q_k+q_{k-1}\leq n<(a_{k+1}+1)q_k+q_{k-1}\} 
$$ 
does not visit any cell in $\Q_{q_k+q_{k-1}}\vee \Q\vee R_\alpha^{-1}\Q$ more than $2$ times. 
\end{corollary}

Applying this result to our modified partition $\Q_n$, we obtain: 
\begin{corollary}\label{cor:classical} 
The orbit segment 
$$ 
\{-n\alpha\colon q_k+q_{k-1}\leq n<(a_{k+1}+1)q_k+q_{k-1}\} 
$$ 
does not visit any cell in the partition $\Q_{q_k+q_{k-1}}$ more than $4$ times. 
\end{corollary}

We are now ready to construct our language stable shift.  The basic idea is to fix an irrational $\alpha$ with its associated partition determined by its continued fraction convergents, and then use an increasing sequence of reals $\beta_j$ such that the codings of these reals stay close to the coding of $\alpha$ for long intervals.  Using Lemma~\ref{lem:lang-reduction} we can replace the coding of each $\beta_j$ with respect to its associated partition determined by its continued fraction convergents by that of $\alpha$, controlling the number of times orbits visit a particular cell using Corollary~\ref{cor:classical-old}.  
For the usual continued fraction expansion and associated partition, the orbit of $\alpha$ visits each cell at most twice, but our count in Corollary~\ref{cor:classical} differs from this standard result, as our partition $\Q_n$ does not include the cells determined by $\Q\vee R_\alpha^{-1}\Q$.  
However, the partition $\Q_n\vee\Q\vee R_\alpha^{-1}\Q$ only has three more cells than $\Q_n$, 
and so our construction carries through with visits of the orbit to any particular cell inflated by at most $3$. 

More precisely, we fix $\alpha\notin\mathbb Q$ with continued fraction expansion
$\alpha=[0,a_1,a_2,\dots]$ and convergents $p_k/q_k$.  
Choose $0<\beta_1<\alpha$.  
Suppose we have chosen real numbers $\beta_1<\beta_2<\dots<\beta_i<\alpha$ and integers $k_1<k_2<\dots<k_i$ such that for all $1\leq j\leq i$, we have 
	\begin{enumerate} 
	\item $(a_{k_j+1}+1)q_{k_j}>44j^2$; \label{condition1}
	\item the real numbers $\beta_j, \beta_{j+1}, \beta_{j+2},\dots,\beta_i$ all lie in the same cell of the partition 
	$\Q_{(a_{k_j+1}+1)q_{k_j}+q_{k_j-1}-1}$ that $\alpha$ lies in; \label{condition2}
	\item the real numbers $-\beta_j, -\beta_{j+1}, -\beta_{j+2},\dots,-\beta_i$ also  all lie in the same cell of the partition 
	$\Q_{(a_{k_j+1}+1)q_{k_j}+q_{k_j-1}-1}$ that  $\alpha$ lies in. \label{condition3}
	\end{enumerate} 
Then for any $q_k+q_{k-1}\leq n<(a_{k+1}+1)q_k+q_{k-1}$, we have that $\beta_j,\beta_{j+1},\dots,\beta_i$ 
all lie in the same cell of $\Q_n$ as each other, and similarly, 
we have that $-\beta_j,-\beta_{j+1},\dots,-\beta_i$ also all lie 
in the same cell of $\Q_n$ as each other.  
We now check when  new minimal forbidden words arise 
for $n$ in the interval  $[q_{k_{j}}+q_{k_{j}-1}, (a_{k_{j}}+1)q_{k_{j}}+q_{k_{j}-1}]$ for $j\in\{1, 2, \ldots, i\}$. 
By Lemma~\ref{lem:lang-reduction}, these words arise only when $-n\alpha$ visits one of $11$ cells from the partition $\Q_{q_{k_{j}}+q_{k_{j}-1}}$ and by 
Corollary~\ref{cor:classical} each can be visited at most $4$ times.  Thus for any $j_1,j_2\in\{1,2,\dots,i\}$ 
there are at most $44$ 
 values of $n$ in the interval $q_{k_{j_1}}+q_{k_{j_1}-1}\leq n<(a_{k_{j_1}}+1)q_{k_{j_1}}+q_{k_{j_1}-1}$ for which $X_{\beta_{j_2}}$ has a minimal forbidden word of length $n+1$.  But combining 
these results with conditions~\eqref{condition2} and~\eqref{condition3}, we also 
have  that for any $1\leq j_1\leq i$, the set of $n$ in the interval 
\begin{equation}
\label{eq:interval}
[q_{k_{j_1}}+q_{k_{j_1}-1},(a_{k_{j_1}}+1)q_{k_{j_1}}+q_{k_{j_1}-1}]
\end{equation}
for which there exists $1\leq j_2\leq i$ such that $X_{\beta_{j_2}}$ has a minimal forbidden word of length $n+1$ is at most $44j_1$.  By condition~\eqref{condition1} of the construction, 
the interval in~\eqref{eq:interval} has length at least $44j_1^2$ 
and so there must be a subinterval of length $j_1$ on which none of the shifts $X_{\beta_1},X_{\beta_2},\dots,X_{\beta_i}$ have any minimal forbidden words.  
Since the sequence $\{q_k\}_{k=1}^{\infty}$ is nondecreasing and tends to infinity, we can choose $k_{i+1}$ 
sufficiently large such that $(a_{k_{i+1}+1}+1)q_{k_{i+1}}>44(i+1)^2$ and choose $\beta_{i+1}\in(\beta_i,\alpha)$ such 
that $\beta_{i+1}$ lies in the same cell of $\Q_{(a_{k_{i+1}+1}+1)q_{k_{i+1}}+q_{k_{i+1}-1}-1}$ as $\alpha$, and 
$-\beta_{i+1}$ lies in the same cell of $\Q_{(a_{k_{i+1}+1}+1)q_{k_{i+1}}+q_{k_{i+1}-1}-1}$ as $-\alpha$.  Since the partitions $\Q_k$ refine each 
other as $k$ increases, it 
follows that for any $1\leq j_1\leq i+1$ that $\beta_{i+1}$ lies 
in the same cell of $\Q_{(a_{k_{j_1}+1}+1)q_{k_{j_1}}+q_{k_{j_1}-1}-1}$ that $\alpha$ lies in, and $-\beta_{i+1}$ 
lies in the same cell of $\Q_{(a_{k_{j_1}+1}+1)q_{k_{j_1}}+q_{k_{j_1}-1}-1}$ that  $-\alpha$ lies in.  By induction, we construct a sequence 
$$ 
0<\beta_1<\beta_2<\beta_3<\dots<\beta_i<\dots<\alpha 
$$ 
which satisfies conditions~\eqref{condition1},~\eqref{condition2}, and~\eqref{condition3} for all $i,j\geq 1$.  Therefore for any $j$ there is an interval of integers (between $q_{k_j}+q_{k_j-1}$ and $(a_{k_j+1}+1)q_{k_j}+q_{k_j-1}$) of length $j$ such that $X_{\beta_i}$ has no minimal forbidden words of any lengths in that interval, for any $i=1,2,3,\dots$  Consider 
the shift 
\begin{equation}
\label{def:Z}
Z:=\overline{\bigcup_{i=1}^{\infty}X_{\beta_i}}. 
\end{equation}
A word $w\in\{0,1\}^*$ is in the language of $Z$ if and only if there exists $i$ such that $w\in\mathcal{L}(X_{\beta_i})$.  It follows that $Z$ is 
language stable.  

Moreover we claim that 
\begin{equation}
\label{eq:ZminusX} 
Z\setminus\bigcup_{i=1}^{\infty}X_{\beta_i} 
\end{equation} 
is the shift $X_{\alpha}$.  We show this in two steps.  First, suppose $z\in Z$ but $z\notin X_{\beta_i}$ for any $i\geq 1$.  Then there is a sequence $\{m_i\}$ of integers tending to infinity, and 
points $x_i\in X_{\beta_{m_i}}$ such that $\lim_i x_i=z$.  In other words, for each fixed $N\geq 1$ we have $z_j=(x_i)_j$ for all $|j|\leq N$.  
Let $s_i\in S^1$ be a point whose coding with respect to the partition $\{[0,\beta_{m_i}),[\beta_{m_i},1)\}$ agrees with $(x_i)_j$ for all $|j|\leq N$.  Passing to a subsequence if necessary, we can assume there exists $y\in S^1$ such that $\lim s_i=y$.  Note that $R_{\alpha}^js_i$ only codes differently with respect to the partitions $\{[0,\beta_{m_i}),[\beta_{m_i},1)\}$ and $\{[0,\alpha),[\alpha,1)\}$ if $R_{\alpha}^js_i\in[\beta_{m_i},\alpha)$.  
Since $\lim_i s_i=y$ and $\lim_i \beta_{m_i}=\alpha$, for fixed $N$ we have that for all sufficiently large $i$, the rotation 
$R_{\alpha}^js_i$ codes in the same way with respect to both of these partitions for all $|j|\leq N$ unless $R_{\alpha}^jy=\alpha$ for some $|j|\leq N$.  Assuming first that  $R_{\alpha}^jy\neq\alpha$ for any $|j|\leq N$, note that $R_{\alpha}^js_i$ and $R_{\alpha}^jy$ code in the same way with respect to both $\{[0,\beta_{m_i}),[\beta_{m_i},1)\}$ and $\{[0,\alpha),[\alpha,1)\}$ for all 
$|j|\leq N$ and all sufficiently large $i$.  Therefore the coding of $R_{\alpha}^jy$ is $z_j$ for all $|j|\leq N$.  
Since this holds for all $N\geq 1$, it follows that $z\in X_{\alpha}$.  The other possibility is that $R_{\alpha}^jy=\alpha$ for some $|j|\leq N$.  In this case, note that for fixed $N$, for any sufficiently small $|\gamma|>0$ the points $R_{\alpha}^j(s_i+\gamma)$ and $R_{\alpha}^js_i$ code in the same way with respect to $\{[0,\beta_{m_i}),[\beta_{m_i},1)\}$ for all $|j|\leq N$ and all sufficiently large $i$ (since $\lim s_i=y$).  Thus 
we can reduce to the previous case and again it follows that $z\in X_{\alpha}$.  

By the minimality of $X_{\alpha}$, it follows that either the system in~\eqref{eq:ZminusX}  is either $X_\alpha$ or is empty. 
To show it is nonempty, fix $y\in S^1$ to be a point whose $R_{\alpha}$-orbit does not include $\alpha\in S^1$.  For 
any fixed $N\geq 1$, note that the rotation $R_{\alpha}^jy$ codes in the same way with respect to $\{[0,\beta_i),[\beta_i,1)\}$ and $\{[0,\alpha),[\alpha,1)\}$ for all $|j|\leq N$ and all sufficiently large $i$.  The coding of the $R_{\alpha}$-orbit of $y$ is an element of $X_{\beta_i}$ and therefore this defines a sequence of points in $\bigcup_i X_{\beta_i}$ whose limit is the coding of $y$ with respect to $\{[0,\alpha),[\alpha,1)\}$, meaning there is some element of $X_{\alpha}$ in $Z$.


Furthermore, note that in our construction, 
we can always choose that $$\beta_i\in\{n\alpha\pmod 1\colon n=1,2,\dots\},$$ 
and henceforth we insist on this.  Then since $[0,\beta_i)$ can be written as a union of intervals in $\bigvee_{i=0}^mR_{\alpha}^{-i}\{[0,\alpha),[\alpha,1)\}$ for sufficiently large $m$,  there is a block code $\varphi_i\colon X_{\alpha}\to X_{\beta_i}$.  
By Durand's Theorem~\cite[Corollary 12]{D}, 
such a block code is invertible and so $X_{\alpha}$ is topologically 
conjugate to $X_{\beta_i}$ for all $i$.  Since $X_{\alpha}$ is uniquely ergodic, 
it follows that $Z$ is the union of countably many uniquely ergodic, topologically conjugate subshifts.  In particular it has only countably many ergodic measures and they are all (measurably) isomorphic to each other. 


Summarizing this construction and using the properties shown in Lemma~\ref{lemma:def-X}, we have: 
\begin{corollary}\label{rem:zi}
Fix some irrational $\alpha\in (0,1)$.  
There exists an increasing sequence of reals $\{\beta_i\}$  with each $\beta_i\in\{n\alpha\pmod 1\colon n\geq 1\}$ 
such that $\lim_i\beta_i = \alpha$ and such that the system 
$$ 
Z:=\overline{\bigcup_{i=1}^{\infty} X_{\beta_i}} = Z_c\cup\bigcup_{j=1}^{\infty}Z_j, 
$$ 
where $Z_j:=X_{\beta_j}$  is the coding of the rotation $([0,1),R_{\alpha})$ with respect to the partition $\{[0,\beta_j),[\beta_j,1)\}$ for all $j=1,2,\dots$ and $Z_c:=X_{\alpha}$ is the coding with respect to $\{[0,\alpha),[\alpha,1)\}$, 
is a language stable subshift.  
Moreover, if $X$ is the system defined in 
 Lemma~\ref{lemma:def-X}, then $X\times Z_j$ is minimal for all $j\in\N$ and $X\times Z_c$ is minimal.
Furthermore, the system 
$$ 
Z\setminus\bigcup_{i=1}^{\infty} X_{\beta_i} = Z\setminus\bigcup_{j=1}^\infty Z_j
$$ 
is the Sturmian shift $X_{\alpha} = Z_c$. 
\end{corollary}

\subsection{The example and its properties}
\label{sec:example-summary}
The example is the shift $X\times Y\times Z$. 

\subsubsection{Language stability and the existence of a characteristic measure of maximal entropy}
We first check that our example is language stable and deduce that it has a characteristic measure.

\begin{lemma}
The shift $(W, \sigma)$, where $W = X\times Y\times Z$ and $X$ is defined as in Lemma~\ref{lemma:def-X}, $Y$ is the 
full $2$-shift, and $Z$ is defined in~\eqref{def:Z}, 
is language stable. 
\end{lemma} 
\begin{proof}
For any $n\in\N$,  we have that $\mathcal{L}_n(X\times Y\times Z)=\mathcal{L}_n(X)\times\mathcal{L}_n(Y)\times\mathcal{L}_n(Z)$.  Therefore $(w_X,w_Y,w_Z)\in\{0,1\}^n\times\{0,1\}^n\times\{0,1\}^n$ is forbidden if and 
only if at least one of the component words is forbidden, meaning that at 
least one of the following holds: $w_X$ is forbidden in $X$, $w_Y$ is forbidden in $Y$, or $w_Z$ is forbidden in $Z$.  Similarly, the word $(w_X,w_Y,w_Z)$ is minimal and forbidden if and only if at least one of $w_X$, $w_Y$, and $w_Z$ is a
forbidden word in its respective shift and none of the $w_X$, $w_Y$, and $w_Z$ is forbidden and not a minimal forbidden word.   
In particular, if each of $X$, $Y$, and $Z$ have no minimal forbidden words of length $n$, 
then $W$ also has no minimal forbidden words of length $n$.  
We claim that there are arbitrarily long intervals of the form $\{N,N+1,N+2,\dots,N+k-1\}$ for which $W=X\times Y\times Z$ has no minimal forbidden words of any length in the interval. 

Since $Y$ is the full $2$-shift, it does not introduce any forbidden words. 
Let 
$\{Z_n\}_{n=1}^{\infty}$ be the SFT cover of the shift $Z$ defined by~\eqref{def:Z} and let 
$$ 
\{i_1,i_1+1,\dots,i_1+j_1\},\{i_2,i_2+1,\dots,i_2+j_2\},\dots,\{i_n,i_n+1,\dots,i_n+j_n\},\dots 
$$ 
be a sequence of intervals of consecutive integers for which the sequence 
$(j_n)_{n\in\N}$ is strictly increasing and such that $Z_{i_n}=Z_{i_n+j_n}$ for all $n$ (this is possible since $Z$ is language stable).  By Lemma~\ref{lemma:def-X} we can construct $X$ such that for infinitely many $n$ we also have $X_{i_n}=X_{i_n+j_n}$.  Therefore there are infinitely many $n$ such that none of $X$, $Y$, or $Z$ has a minimal forbidden word of any length in the interval $\{i_n,i_n+1,\dots,i_n+j_n\}$, and so $X\times Y\times Z$ also has no minimal forbidden word in any such interval.  Since the sequence $(j_n)_{n\in\N}$ is strictly increasing, $W:=X\times 
Y\times Z$ is language stable. 
\end{proof} 

Combining this with Corollary~\ref{cor:language-stable}, we have that the 
existence of a characteristic measure on this system:
\begin{corollary}
The system $(W, \sigma)$ has a characteristic measure.  Moreover it has a 
characteristic measure that is a measure of maximal entropy.
\end{corollary}

We next check that the existence of this characteristic measure for the system  $(W,\sigma)$ does not follow from 
results already previously in the literature. 

\subsubsection{Showing that Lemma~\ref{lem:measurable} does not apply to this system}
\begin{lemma} 
If $\mu$ is any characteristic measure on $(W, \sigma)$, then the set 
$$ 
\{\nu\in\mathcal{M}(X)\colon(X,\sigma,\nu)\text{ is measurably isomorphic 
to }(X,\sigma,\mu)\} 
$$ 
is infinite.  In particular, Lemma~\ref{lem:measurable} does not apply to 
$W = X\times Y\times Z$.
\end{lemma} 
\begin{proof} 
Let $\mu$ be a characteristic measure on  $(W, \sigma)$.  Let $\mu_{YZ}$ be the marginal measure obtained by projecting $\mu$ 
onto $Y\times Z$, meaning that for any measurable $A\subseteq Y\times Z$ we have $\mu_{YZ}(A):=\mu(X\times A)$.  Note that $\mu_{YZ}$ is a shift 
invariant probability measure on $Y\times Z$. 
Next set  $\mu_Y$ to be the marginal of $\mu_{YZ}$ projected onto $Y$ and 
set $\mu_Z$ to 
be the marginal of $\mu_{YZ}$ projected onto $Z$; thus for measurable $B\subseteq Y$ and $C\subseteq Z$, 
we have $\mu_Y(B):=\mu_{YZ}(B\times Z)$ and 
$\mu_Z(C):=\mu_{YZ}(Y\times C)$.  Then $\mu_Y$ is an invariant measure on $Y$ and $\mu_Z$ is an invariant measure on $Z$.  

If $\varphi\in\Aut(Y)$ is an automorphism of $Y$ and if $B\subseteq Y$ is 
a measurable set, then (letting $\id$ denote the identity) we have 
\begin{eqnarray*} 
\mu_Y(\varphi^{-1}B)&=&\mu_{YZ}((\varphi\times\id)^{-1}(B\times Z)) \\ 
&=&\mu((\id\times\varphi\times\id)^{-1}(X\times B\times Z)) =\mu(X\times B\times Z) \\ 
&=&\mu_{YZ}(B\times Z) =\mu_Y(B), 
\end{eqnarray*} 
where the third equality holds because $\id\times\varphi\times\id\in\Aut(X\times Y\times Z)$ and $\mu$ is characteristic.  Therefore, $\mu_Y$ is an $\Aut(Y)$-characteristic measure on the full shift $Y$.  
Thus by Lemma~\ref{lem:full-shift}, we can decompose the measure $\mu_Y$ 
into a convex combination of the symmetric Bernoulli measure on $Y$ and atomic measures supported on unions of periodic orbits in $Y$, writing 
\begin{equation}
\label{eq:decomp-muY}
\mu_Y=c_0\mu_B+\sum_{i=1}^{\infty}c_i\mu_{p_i} 
\end{equation}
where $0\leq c_i\leq 1$ for all $i$, $\sum_{i=0}^{\infty}c_i=1$, $\mu_B$ is the symmetric Bernoulli measure on $Y$, $p_i$ is a collection of pairwise disjoint unions of periodic orbits, and $\mu_{p_i}$ is a characteristic measure supported on $p_i$.  

Furthermore, the measure $\mu_{YZ}$ is a joining of the measures $\mu_Y$ and $\mu_Z$.  
But $h_{\tp}(Z)=0$ and  $h(\mu_Z)=0$.  Therefore, the measure $\mu_B$ 
(in the decomposition~\eqref{eq:decomp-muY} of $\mu_Y$) is disjoint from $\mu_Z$.  Moreover, $\mu_Z$ is an invariant measure on $Z$ and so it is an at most countable convex combination of ergodic measures that are all isomorphic to the same irrational circle rotation.  
It follows that $\mu_Z$ is disjoint from every finite rotation, and in particular from $\mu_{p_i}$ for all $i$. 
Combining these two observations, it follows that $\mu_Y$ and $\mu_Z$ are 
disjoint, 
and so we have that $\mu_{YZ}=\mu_Y\times\mu_Z$.  

We write the decomposition of $\mu_Z$ as 
$$ 
\mu_Z=\sum_{i=1}^{\infty}d_i\mu_{Z,i}, 
$$ 
where $\mu_{Z,i}$ is an enumeration of the countably many ergodic measures supported on $Z$ with weights $0\leq d_i\leq 1$ satisfying $\sum_{i=1}^{\infty}d_i=1$.  For each integer $k\geq 0$, the measure $\mu_Z^k$ defined by 
$$ 
\mu_Z^k:=\sum_{i=1}^{\infty}d_{i+k}\mu_{Z,i} 
$$ 
is measurably isomorphic to $\mu_Z$, and the resulting 
measures $\{\mu_Z^k\}_{k=0}^{\infty}$ are pairwise distinct.  Therefore, for each $k \geq 0$, the measure $\mu_{YZ}=\mu_Y\times\mu_Z$ is measurably isomorphic to each of the pairwise distinct measures $\mu_Y\times\mu_Z^k$.  

Finally, set $\mu_X$ to be the marginal obtained by projecting $\mu$ onto 
$X$, 
and so for measurable $A\subseteq X$ we have $\mu_X(A):=\mu(A\times Y\times Z)$.  Then $\mu$ is a joining of $\mu_X$ with $\mu_{YZ}$.  Since $\mu_{YZ}$ is isomorphic to $\mu_Y\times\mu_Z^k$, there is a joining of $\mu_X$ 
with $\mu_Y\times\mu_Z^k$ that is isomorphic to $\mu$ (and distinct from it since it has a different marginal onto $Y\times Z$).  These joinings are pairwise distinct for each $k\geq 0$, so $\mu$ is measurably isomorphic to infinitely many measures supported on  $(W, \sigma)$. 

Note that if $\mu$ were the measure produced by Lemma~\ref{lem:measurable} then $\mu$ would only be measurably isomorphic to finitely many other measures on  $(W, \sigma)$.  Therefore $\mu$ does not result from applying 
Lemma~\ref{lem:measurable} to any measure on  $(W, \sigma)$.  Since $\mu$ 
is an arbitrary characteristic measure on $(W, \sigma)$, Lemma~\ref{lem:measurable} cannot be applied to this system. 
\end{proof} 

\subsubsection{Showing the Krylov-Bogolioubov Theorem does produce a characteristic measure on this system}
Next we check that the automorphism group of $(W, \sigma)$ is not amenable,  meaning that we can not apply
the Krylov-Bogolioubov Theorem to produce a characteristic measure for this system. 

\begin{lemma}
The automorphism group of $(W, \sigma)$ is not amenable (as a countable discrete group).
\end{lemma} 

\begin{proof} 
The automorphism group of $Y$ is nonamenable, since $Y$ is a full-shift on at least two symbols (see~\cite{BLR}).  For each $\varphi\in\Aut(Y)$, the map $\varphi\mapsto\id\times\varphi\times\id$ gives an embedding of $\Aut(Y)$ into $\Aut(W)$.  
Since any subgroup of an amenable, countable discrete group is also amenable, it follows that 
$\Aut(W)$ is also nonamenable.  
\end{proof} 

\subsubsection{Showing that $W$ has no zero entropy subsystems} 
We check that Frisch-Tamuz's Theorem~\cite{FT2} that zero entropy subshifts have characteristic measures cannot be used to find a characteristic measure on $W$.

\begin{lemma}
Every subsystem of $(W,\sigma)$ has positive entropy.
\end{lemma}
\begin{proof}
Since $X$ is minimal and has positive entropy, every subsystem of $W=X\times Y\times Z$ has topological entropy at least $h_{\tp}(X)>0$.
\end{proof}

\subsubsection{Showing that Lemma~\ref{lem:topological}  cannot be used to find a characteristic measure of maximal entropy on $W$}  Finally, we show that the characteristic measure of maximal entropy on $W$ that is guaranteed by language stability cannot be obtained by applying Lemma~\ref{lem:topological}.  We begin with some results characterizing the full entropy, proper subshifts of $W$.

\begin{lemma}\label{lem:new} 
Let $W^{\prime}\subseteq W$ be a closed subshift with topological entropy 
$h_{\tp}(W)$.  Then there exists a nonempty set $S\subseteq\N\cup\{c\}$ such that 
$$ 
X\times Y\times\left(\bigcup_{j\in S}Z_j\right)\subseteq W^{\prime}, 
$$ 
where $Z_j$ is defined as in Corollary~\ref{rem:zi}.  Moreover, $S$ can be chosen such that for any $j\notin S$ the shift  
$$ 
(X\times Y\times Z_j)\cap W^{\prime}
$$ 
has topological entropy strictly lower than $h_{\tp}(W^{\prime})$.  Finally, if $S$ is chosen in this way and is infinite, then $c\in S$.  
\end{lemma}

\begin{proof}
Let $W^{\prime}\subseteq W=X\times Y\times Z$.  We define 
\begin{eqnarray*} 
\proj_X(W^{\prime})&:=&\{x\in X\colon\text{there exist $y\in Y$ and $z\in Z$ such that }(x,y,z)\in W^{\prime}\}; \\ 
\proj_Y(W^{\prime})&:=&\{y\in Y\colon\text{there exist $x\in X$ and $z\in Z$ such that }(x,y,z)\in W^{\prime}\}; \\ 
\proj_Z(W^{\prime})&:=&\{z\in Z\colon\text{there exist $x\in X$ and $y\in Y$ such that }(x,y,z)\in W^{\prime}\}. 
\end{eqnarray*} 
Note that $\proj_X(W^{\prime})\subset X$, $\proj_Y(W^{\prime})\subseteq Y$, and $\proj_Z(W^{\prime})\subseteq Z$ are each closed subshifts.  By minimality of $X$, $\proj_X(W^{\prime})=X$.  It is a classical result that any full shift, in particular $Y$, is entropy minimal, meaning every proper subshift of it has entropy strictly lower than $h_{\tp}(Y)$.  Since 
$$ 
h_{\tp}(W^{\prime})\leq h_{\tp}(\proj_X(W^{\prime}))+h_{\tp}(\proj_Y(W^{\prime}))+h_{\tp}(\proj_Z(W^{\prime})) 
$$ 
and $h_{\tp}(W^{\prime})=h_{\tp}(W)=h_{\tp}(X)+h_{\tp}(Y)+h_{\tp}(Z)$, we must have $\proj_Y(W^{\prime})=Y$.  Finally $Z$ is a countable union of minimal subshifts 
$$ 
Z=\bigcup_{j\in\N\cup\{c\}}Z_j 
$$ 
and so there exists a subset $S^{\prime}\subseteq\N\cup\{c\}$ such that 
$$ 
\proj_Z(W^{\prime})=\bigcup_{j\in S^{\prime}}Z_j. 
$$ 
Suppose $S^{\prime}$ is this subset.  

By construction (see Corollary~\ref{rem:zi}), $X\times Z_j$ is minimal for each $j\in S^{\prime}$.  Therefore 
$$ 
\proj_{X\times Z}(W^{\prime}):=\{(x,z)\in X\times Z\colon\text{there exists $y\in Y$ with }(x,y,z)\in W^{\prime}\}
$$ 
is the subshift $X\times\bigcup_{j\in S^{\prime}}Z_j$.  For each $j\in S$ 
define $Y_j\subseteq Y$ by 
$$ 
Y_j:=\{y\in Y\colon\text{there exists }(x,z)\in X\times Z_j\text{ such that }(x,y,z)\in W^{\prime}\}. 
$$ 
Since $\proj_Y(W^{\prime})=Y$, we have $Y=\bigcup_{j\in S^{\prime}}Y_j$.  By the Baire Category Theorem, $Y$ cannot be written as the union of 
a countable number of proper subshifts of itself, so there is a non-empty 
set $S\subseteq S^{\prime}$ such that $Y_j=Y$ for all $j\in S$.  For any such $j$, by a theorem in Furstenberg~\cite[Theorem II.2]{furstenberg}, the set $W^{\prime}\cap(X\times Y\times Z_j)=X\times Y\times Z_j$ (in Furstenberg's terminology, 
a Bernoulli flow, such as $Y$, is disjoint from a minimal flow,  such as $X\times Z_j$).  Therefore 
$$ 
X\times Y\times\left(\bigcup_{j\in S}Z_j\right)\subseteq W^{\prime}. 
$$ 
Moreover, for any $j\notin S$ we have $Y_j\neq Y$ and so $h_{\tp}(Y_j)<h_{\tp}(Y)$ and the topological entropy of $(X\times Y\times Z_j)\cap W^{\prime}$ is at most
$$h_{\tp}(X)+h_{\tp}(Y_j)+h_{\tp}(Z_j)<h_{\tp}(X)+h_{\tp}(Y)+h_{\tp}(Z)=h_{\tp}(W^{\prime}).$$

Finally we show that if $S$ is infinite, then $c\in S$.  Recall that in the proof of Corollary~\ref{cor:classical} we showed that $Z_c=Z\setminus\bigcup_{j=1}^{\infty}Z_j$.  The proof of this relies only on the fact 
that $Z$ is the closure of the union of the infinite collection of shifts $Z_j$.  If $S$ is infinite, then the same proof as in Corollary~\ref{cor:classical} shows that $Z_c=\overline{\bigcup_{j\in S\setminus\{c\}}Z_j}\setminus\bigcup_{j\in S\setminus\{c\}}Z_j$. Since 
$$ 
X\times Y\times\left(\bigcup_{j\in S}Z_j\right)\subseteq W^{\prime}
$$ 
and $W^{\prime}$ is closed, it follows that $c\in S$.
\end{proof}

We recall a result that follows quickly from Boyle, Lind, and Rudolph~\cite[Proposition 9.4]{BLR}.
\begin{lemma}\label{lem:blr}
Let $Y^{\prime}\subseteq Y$ be an infinite proper subshift.  Then $Y^{\prime}$ is topologically conjugate to infinitely many other subshifts of $Y$.
\end{lemma}
\begin{proof}
Taking $X_T:=Y$ in Proposition 9.4 of~\cite{BLR}, if $Y^{\prime\prime}$ 
is any mixing subshift with $Y^{\prime}\subseteq Y^{\prime\prime}\subseteq Y$, then $Y^{\prime\prime}$ lies 
in the $\Aut(Y)$-orbit closure of $Y^{\prime}$ in the Hausdorff metric (note that the use of this result 
requires that $Y^{\prime}$ is infinite).  So if $Y^{\prime}$ is contained in 
infinitely many mixing subshifts of $Y$, then the $\Aut(Y)$-orbit of $Y^{\prime}$ cannot be finite.  Let $\mathcal{F}$ be a minimal list of forbidden words defining $Y^{\prime}$.  Since $Y^{\prime}\neq Y$, we know that $\mathcal{F}\neq\emptyset$.  Let $w\in\mathcal{F}$.  For each $n\in\N$, the subshift $Y^{\prime\prime}_n$ whose only forbidden word is $w10^n1$ contains $Y^{\prime}$ and is contained in $Y$ and for any $u,v\in\mathcal{L}(Y^{\prime\prime}_n)$ we have $u0^{|w|+n+2+t}v\in\mathcal{L}(Y^{\prime\prime}_n)$ for any $t\geq1$.  In particular, $Y^{\prime\prime}_n$ is mixing.  Since $Y^{\prime\prime}_n\neq Y^{\prime\prime}_m$ for any $n\neq m$, the $\Aut(Y)$-orbit of $Y$ cannot be finite.  But every element of the $\Aut(Y)$-orbit of $Y^{\prime}$ is topologically conjugate to $Y^{\prime}$.
\end{proof}

\begin{proposition}\label{prop:no-top}
Let $W^{\prime}\subseteq W$ be a closed, proper subshift $W$ with entropy $h_{\tp}(W)$. 
 Then $W^{\prime}$ is topologically conjugate to infinitely many other subshifts of $W$.
\end{proposition}
\begin{proof}
By Lemma~\ref{lem:new} there is a nonempty set $S\subseteq\N\cup\{c\}$ such that 
$$ 
X\times Y\times\left(\bigcup_{j\in S}Z_j\right)\subseteq W^{\prime}
$$ 
and for every $j\notin S$, $(X\times Y\times Z_j)\cap W^{\prime}$ has entropy strictly lower than $h_{\tp}(W)$.  
Since $W^{\prime}\neq W$, we have  that $S\neq\N\cup\{c\}$.    
We consider two cases: when $S$ is infinite and when $S$ is finite.

If $S$ is infinite, then $c\in S$ by Lemma~\ref{lem:new}.  In this case, let 
$$ 
Y_j:=\{y\in Y\colon\text{there exist }(x,z)\in X\times Z_j\text{ such that }(x,y,z)\in W^{\prime}\}. 
$$ 
If $Y_j=Y$, 
then $(X\times Y\times Z_j)\cap W^{\prime}=X\times Y\times 
Z_j$ is a subshift of entropy $h_{\tp}(W)=h_{\tp}(W^{\prime})$ (recall that $Z=\bigcup_{j\in\N\cup\{c\}}Z_j$ has entropy $0$).  So $Y_j\neq Y$ 
for all $j\notin S$.  
Fix some $j\notin S$.  Since $c\in S$, it follows that $j\neq c$. 
By construction, $Z_j$ is the coding of $([0,1),R_{\alpha})$ with respect to the partition $\{[0,\beta_j),[\beta_j,1)\}$, 
where $\beta_j\in(0,1)$ is a fixed real number distinct from all other $\beta_k$ and the increasing sequence of $\beta_j$ satisfies 
$\lim_k\beta_k = \alpha$.  
Again by construction, $\beta_j$ is not an accumulation point of $\{\beta_k\colon k\neq j\}$,  
and so it follows from the unique ergodicity of the system  $Z_j$ 
that 
there exists $N\geq 1$ such that for all $n\geq N$ we have 
\begin{equation}
\label{eq:choose-N}
\mathcal{L}_n(Z_j)\cap\mathcal{L}_n(Z_k)=\emptyset \text{ for all } k\neq j \quad \text{ and } \quad \mathcal{L}_n(Z_j)\cap\mathcal{L}_n(Z_c)=\emptyset. 
\end{equation}

Note that for any $k\in S$, $X\times Y_j\times Z_j$ is topologically conjugate $X\times Y\times Z_k$, by Durand's Theorem~\cite{D}.  
Fixing $k\in S$ and using that $Z_j$ is topologically conjugate to $Z_k$, by the Curtis-Hedlund-Lyndon Theorem we can 
choose an invertible block code $\psi$ that implements this conjugacy. 
 Let $\psi^{-1}$ be its inverse block code.  
 Choose $n\geq 1$ sufficiently large 
 such that $\mathcal{L}_n(Z_j)$ and $\mathcal{L}_n(Z_k)$ are both disjoint from 
$\mathcal{L}_n(X_t)$ for all $t\notin\{j,k\}$.  
Let $\Psi$ be an invertible block code of range $\max\{\text{range}(\psi),\text{range}(\psi^{-1}),n,N\}$ that acts like $\psi$ on $Z_j$, acts like $\psi^{-1}$ on $Z_k$, and acts like the identity on $Z_t$ for $t\notin\{j,k\}$.  The image of $W^{\prime}$ under the block code $\id\times \id\times\Psi$ is topologically conjugate to $W^{\prime}$, and is a subshift of $W$.  Moreover for fixed $j\notin S$ we can do this construction for any of the infinitely many $k\in S\setminus\{c\}$ and obtain distinct images for distinct $k$: namely, when this construction is carried out with parameters $j$ and $k$, the  entropy of the resulting
subshift intersected with $X\times Y\times Z_t$ is larger than that of 
$W^{\prime}\cap(X\times Y\times Z_t)$ if and only if $t=j$, and has entropy smaller than that of  $W^{\prime}\cap(X\times Y\times Z_t)$ if and only if $t=k$.  
Thus $W^{\prime}$ is topologically conjugate to infinitely many other subshifts of $W$ 
when $S$ is infinite. 

If $S$ is finite but $S\neq\{c\}$, then $(\mathbb{N}\cup\{c\})\setminus S$ is infinite, 
and we can use the same construction, taking the same $j\in S\setminus\{c\}$ 
and using the infinitely many $k\in(\mathbb{N}\setminus S)$ to obtain infinitely many subshifts of $W$ that are topologically conjugate to $W^{\prime}$. 

Finally we consider the case when $S=\{c\}$.  If there exists $j\notin S$ such that $Y_j$ is infinite, then by Lemma~\ref{lem:blr} there are infinitely many other subshifts of $Y$ that are topologically conjugate to $Y_j$.  Let $\varphi$ be a block code implementing any such conjugacy.  
Choose $N$ as in~\eqref{eq:choose-N}.   Then $\id\times\varphi\times \id$ can be implemented by a block code of range at least $N$ and this block code can be extended to a block code that acts like the identity on $X\times Y_k\times Z_k$ for all $k\neq j$.  The image of $W^{\prime}$ under this subshift is topologically conjugate to $W^{\prime}$ and is a subshift of $W$.  Moreover for any two distinct subshifts topologically conjugate to $Y_j$, the 
resulting subshifts of $W^{\prime}$ are distinct, as the projections onto the middle coordinates are distinct.  In this case, $W^{\prime}$ is topologically conjugate to infinitely many subshifts of $W$.  
Thus it suffices considering the case when $Y_j$ is a finite shift for all $j\notin S$.  

If $|\{|Y_j|\colon j\notin S\}|>1$, then there exists $j\in(\mathbb{N}\setminus S)$ such that there are infinitely many $k\in(\mathbb{N}\setminus S)$ such that $|Y_j|\neq|Y_k|$. As before, we can find a block code that acts like the identity on $W^{\prime}\cap(X\times Y_t\times Z_t)$ for all $t\notin\{j,k\}$ 
and acts like $\id\times \id\times\Psi$ on $(X\times Y\times Z_j)\cup(X\times Y\times Z_k)$,  and $\psi$ is a topological conjugacy between $Z_j$ and $Z_k$.  As there are infinitely many choices for $k$, there are again infinitely many subshifts of $W$ topologically conjugate to $W^{\prime}$.    Thus the final case remaining 
is when $S=\{c\}$ and $|Y_i|=|Y_j|$ for all $i,j\in\mathbb{N}$.  In this case, fix $j\in\mathbb{N}$.  Note that since $Y_j$ is a finite, invertible system, it is the disjoint union of a finite number of cycles.  Therefore $Y_j\times Z_j$ decomposes into the disjoint union of a finite number of systems of the form $\{0,1,\dots,k-1\}\times Z_j$ where the map on $\{0,1,\dots,k-1\}$ is addition modulo $k$.  Applying Lemma~\ref{lem:low-entropy} to each of these components individually, there exists a subshift $T_j\subseteq\{0,1\}^{\Z}$ that is topologically conjugate to $Y_j\times Z_j$.  Let $\varphi_j\colon T_j\to Y_j\times Z_j$ be this conjugacy and let $\pi\colon Y_j\times Z_j\to Z_j$ denote projection onto the second coordinate.  Then the shift 
$$ 
U_j:=X\times\{(t,\pi(\varphi_j(t)))\colon t\in T_j\}\subseteq X\times Y\times Z_j 
$$ 
is topologically conjugate to $X\times Y_j\times Z_j$.  
Since $T_j$ is infinite, these shifts are distinct.  Now, as before, we can find a block code that acts like the identity on $W^{\prime}\cap(X\times Y_t\times Z_t)$ for all $t\neq j$ and acts like a conjugacy between $X\times Y_j\times Z_j$ and $U_j$ on $X\times Y_j\times Z_j$.  Since there are infinitely many choice for $j$, again $W^{\prime}$ is topologically conjugate to finitely many subshifts of $W$.
\end{proof}

\begin{proposition}
The characteristic measure of maximal entropy $\mu$ on $W$, guaranteed to exist by Corollary~\ref{cor:language-stable}, cannot be obtained by applying Lemma~\ref{lem:topological} to any proper subshift $W^{\prime}\subseteq W$. 
\end{proposition}
 In other words, this proposition shows that only way to obtain $\mu$ from Lemma~\ref{lem:topological} is to already know of a characteristic measure of maximal entropy on $W$.
 
\begin{proof}
For contradiction, suppose $\mu$ can be obtained by applying Lemma~\ref{lem:topological}  to some proper subshift $W^{\prime}\subseteq W$.  Then $\mu$ is a convex combination of a finite number of measures, each of which is supported on a subshift topologically conjugate to $W^{\prime}$.  Therefore $h_{\mu}(\sigma)\leq h_{\tp}(W^{\prime})$.
 But $\mu$ is a measure of maximal entropy on $W$ and so $h_{\tp}(W^{\prime})=h_{\tp}(W)$.  Since Lemma~\ref{lem:topological} only applies to a shift $W^{\prime}$ that is only topologically conjugate to a finite number of other subshifts of $W$, and since any proper subshift of $W$ with entropy $h_{\tp}(W)$ is topologically conjugate to infinitely many other subshifts of $W$ by Lemma~\ref{prop:no-top}, it follows that $W^{\prime}=W$.  But this is contradiction of our assumption that $W^{\prime}$ is a proper subshift of $W$.  
\end{proof}

\section{An application}
\label{sec:application}
\subsection{When does a block code define an automorphism?}
Suppose $X\subseteq\A^{\Z}$ is a subshift and $\varphi\colon\mathcal{L}_{2R+1}(X)\to\A$ is a block code.  A natural question is whether $\varphi$ defines an automorphism of $X$, and answering this question in general is 
challenging. 
We note, however, that if $\mu$ is a characteristic measure for $(X,\sigma)$, then $\mu$ gives rise to a family of necessary conditions that must be satisfied if $\varphi$ is an automorphism.  Making this precise, we have: 
\begin{lemma}\label{lem:application} 
Let $X\subseteq\A^{\Z}$ be a subshift and let $\varphi\colon\mathcal{L}_{2R+1}(X)\to\A$ be a range
$R$ block code with $\varphi(X)\subseteq X$.  For each $w\in\mathcal{L}(X)$ define  
$$ 
\varphi^{-1}(w):=\{v\in\mathcal{L}_{2R+|w|}(X)\colon\varphi(v)=w\}. 
$$ 
Suppose $\mu$ is a characteristic measure for $(X,\sigma)$.  If there exists $w\in\mathcal{L}(X)$ such that $\mu([w])\neq\mu([\varphi^{-1}(w)])$, then $\varphi$ does not define an automorphism of $X$. 
\end{lemma}
In other words, in shifts where one can estimate the measure of cylinder sets for a characteristic measure, one can quickly eliminate many block codes for consideration 
as automorphisms of the system.  Here we note that Lemma~\ref{lem:application} generalizes a theorem of Hedlund in the special case of full shifts~\cite[Theorem 5.4]{hedlund}.  We demonstrate how to use Lemma~\ref{lem:application}, in the case of the well-known Fibonacci shift. 

\subsection{Use of Lemma~\ref{lem:application} for the Fibonacci shift} The {\em Fibonacci shift} is the subshift of finite type whose only forbidden word is $11$; it is called the Fibonacci shift because the size of $\mathcal{L}_n(X)$ grows according to the Fibonacci sequence (note this should not be confused with the Fibonacci substitution system, also known as 
the Fibonacci Sturmian shift).  It is not hard to show that it is topologically mixing and so by Parry's theorem~\cite{parry}, it has a unique measure of maximal entropy $\mu$, called its {\em Parry measure}.  As 
already noted, $\mu$ is a characteristic measure for $X$ and so any automorphism on $X$ preserves $\mu$.  The Parry measure of a subshift of finite type can be written down explicitly and a computation following~\cite{parry} (see also~\cite[Section 4.4]{HK}) gives the measures: 
	\begin{align*} 
	\mu([00])&=\frac{b}{b+2}\approx0.44721; \\ \\ 
	\mu([10]) & =  \mu([01]) = \frac{1}{b+2}\approx0.27639; \\ \\ 
	\mu([0000])& =\mu([0010])=\mu([0100])=\frac{a}{b+2}\approx0.17082; 
\\ \\ 
	\mu([0001])& =\mu([0101])=\mu([1000])=\mu([1010])=\frac{a^2}{b+2}\approx0.10557; \\ \\ 
	\mu([1001])&=\frac{a^3}{b+2}\approx0.06525, 
	\end{align*} 
where $a:=2/(\sqrt{5}+1)$ and $b:=2/(\sqrt{5}-1)$.  
As a particular example, let $\varphi$ denote the block code with range $1$ given by:
	\begin{align*} 
	\varphi(000)& = \varphi(001)=\varphi(010)=\varphi(100)=0; \\ 
	\varphi(101)& =1. 
	\end{align*} 
To determine if $\varphi$ defines an automorphism of $X$, 
we first check that $\varphi(X)\subseteq X$ by checking that the image of 
any element of $\mathcal{L}_4(X)$ is an element of $\mathcal{L}_2(X)$ (this guarantees that the forbidden word $11$ does not occur in any element of $\varphi(X)$).  Then we check if the necessary conditions provided by Lemma~\ref{lem:application} are satisfied.  In this case, we conclude that $\varphi$ does {\em not} have an inverse block code because 
\begin{eqnarray*} 
\mu([\varphi^{-1}(00)])&=&\mu([0000]\cup[0001]\cup[0010]\cup[0100]\cup[1000]\cup[1001]) \\
&=&\mu([0000])+\mu([0001])+\mu([0010])+\mu([0100])+\mu([1000])+\mu([1001]) \\ 
&>&\mu([00]). 
\end{eqnarray*} 

Of course for subshifts that are not shifts of finite type, the problem of determining whether $\varphi(X)\subseteq X$ is more challenging.  Further, even in cases when it is known to exist, finding a characteristic measure and explicitly writing down the measure of small cylinder sets is significantly more difficult.  However, we mention that the characteristic measure we construct on language stable shifts is a weak* limit of Parry measures on shifts of finite type (at least in the case when the terms in 
the SFT cover are themselves topologically mixing) and this allows one to 
approximate the measure our characteristic measure would give to a (small) cylinder set by studying the measure it gets in the terms of the SFT cover.

\end{document}